\documentclass[preprint]{amsart}

\usepackage{multicol}
\usepackage{graphicx}
\usepackage{hyperref}
\usepackage{enumerate}
\usepackage{enumitem} 
\usepackage{amsfonts}
\usepackage{amsmath}
\usepackage{amssymb}
\usepackage{amsthm}
\usepackage{thmtools}
\usepackage{mathrsfs}
\usepackage{dsfont}
\declaretheoremstyle[
notefont=\bfseries, notebraces={}{},
bodyfont=\normalfont,
postheadspace=0.5em,
numbered=no,
]{mystyle}

\declaretheoremstyle[
notefont=\bfseries, notebraces={}{},
bodyfont=\normalfont,
postheadspace=0.5em,
numbered=no,
]{mystyle}

\declaretheoremstyle[
notefont=\bfseries, notebraces={}{},
bodyfont=\normalfont,
postheadspace=0.5em,
numbered=no,
]{mystyle}
\declaretheorem[style=mystyle]{Corollary}

\renewcommand\Re{\operatorname{\mathfrak{Re}}}
\renewcommand\Im{\operatorname{\mathfrak{Im}}}
\newcommand{\Cset}{\mathop{\mathds{C}}}
\newcommand{\Nset}{\mathop{\mathds{N}}}
\newcommand{\Zset}{\mathop{\mathds{Z}}}
\newcommand{\Rset}{\mathop{\mathds{R}}}
\newcommand{\Dset}{\mathop{C_{\scriptscriptstyle{\!\Rset\rightarrow\Cset}}^{\infty}}}

\providecommand{\OO}[1]{\operatorname{O}\left(#1\right)}

\newcommand{\s}[1]{{s}_{#1}}
\newcommand{\ds}[1]{{d}_{#1}}

\newcommand{\zt}[1]{\mathop{a_{#1}}}
\newcommand{\ztt}[1]{\mathop{b_{#1}}}

\newcommand{\am}[1]{|\zt{#1}|}
\newcommand{\amt}[1]{|\ztt{#1}|}

\newcommand{\Te}[1]{{\mathop{\Phi_{#1}}}}

\newcommand{\sn}[1]{{\mathop{\Im(\zt{#1})}}}
\newcommand{\cs}[1]{{\mathop{\Re(\zt{#1})}}}
\newcommand{\mult}[1]{{\mathop{\mu(#1)}}}
\newcommand{\snt}[1]{{\mathop{\Im(\ztt{#1})}}}
\newcommand{\cst}[1]{{\mathop{\Re(\ztt{#1})}}}
\newcommand{\del}[1]{{\mathop{\Delta_{{#1}}}}}

\newcommand{\PP}[2]{P_{#1}(#2)}

\newcommand{\K}[1]{\ensuremath{{\mathcal{K}}^{#1}}}

\newcommand{\dx}{\mathrm{d}}
\newcommand{\dd}{{\mathop{{D}}}}

\newcommand{\e}{\mathop{{\mathrm{e}}}}
\newcommand{\ii}{\mathop{\dot{\imath}}}
\newcommand{\CC}{\ensuremath{\mathcal{L}}}
\newcommand{\LIM}[1]{\lim_{#1\rightarrow\infty}}
\newcommand{\g}[1]{\gamma_{#1}}
\newcommand{\res}[2]{\mathrm{Res}\left(#1;\, #2\right)}

\newtheorem{theorem}{Theorem}
\newtheorem{corollary}[theorem]{Corollary}

\newtheorem{conjecture}[theorem]{Conjecture}
\newtheorem{lemma}[theorem]{Lemma}
\newtheorem{definition}{Definition}

\begin{document}


\title[{ASYMPTOTIC BEHAVIOR OF FAMILIES OF ORTHONORMAL POLYNOMIALS}]{Asymptotic behaviour of some families of orthonormal polynomials and an associated Hilbert space}

\author[A.~IGNJATOVIC]{Aleksandar Ignjatovi\'{c }}

\thanks{This paper is dedicated to my wife Sharon Younghi Choi; without her love and patience this work would have never seen daylight. I also want to thank Jeff Geronimo, Doron Lubinsky, Paul Nevai, Vilmos Totik and especially the anonymous referees for their most valuable comments which have greatly improved this article.}

\address{School of Computer Science and Engineering, University of New South Wales, Sydney, Australia; email: ignjat@cse.unsw.edu.au}

\markboth{}{}

\maketitle

\begin{abstract} 
We characterise asymptotic behaviour of families of symmetric orthonormal polynomials whose recursion coefficients satisfy certain conditions, satisfied for example by the (normalised) Hermite polynomials. More generally, these conditions are satisfied by the recursion coefficients of the form $c(n+1)^p$ for $0<p<1$ and $c>0$, as well as by recursion coefficients which correspond to polynomials orthonormal with respect to the exponential weight $W(x)=\exp(-|x|^\beta)$ for $\beta>1$.  We use these results to show that, in a Hilbert space defined  in a natural way by such a family of  orthonormal polynomials, every two complex exponentials $e_{\omega}(t)={\e}^{\ii \omega t}$ and $e_{\sigma}(t)={\e}^{\ii \sigma t}$ of distinct frequencies $\omega,\sigma$ are mutually orthogonal. We finally formulate a surprising conjecture for the corresponding families of  non-symmetric orthonormal polynomials; extensive numerical tests indicate that such a conjecture appears to be true. \\

\noindent\textbf{keywords:} orthogonal polynomials, unbounded recurrence coefficients, Christoffel functions, almost periodic functions, signal processing\\
\textbf{AMS classification numbers:} 42C05,  41A60,  42A75
\end{abstract}


\section{Introduction}Let $\g{n}>0$ for $n\geq 0$  be the recursion coefficients that correspond to a symmetric positive definite family of  orthonormal polynomials $(p_n:n\in\Nset)$. Thus, $p_0(\omega)=1$ and if we set $\g{-1}=1$ and $p_{-1}(\omega)=0$, then the three term recurrence  
\begin{equation}\label{poly}
{\g{n}}p_{n+1}(\omega)={\omega}\, p_{n}(\omega)-{\g{n-1}}\, p_{n-1}(\omega)
\end{equation}
holds for all $n\geq 0$.\footnote{See, for example, \cite{GF}.}
Let also $\s{n}$ be the first and $\ds{n}$ the second order forward finite differences of these recursion coefficients:
\begin{align*}
&\s{n}=\g{n+1}-\g{n}; && \ds{n}=\s{n+1}-\s{n}. &
\end{align*}
We consider families of orthonormal polynomials such that the corresponding recursion coefficients $\gamma_n$ satisfy the following conditions. %
\newlist{UR}{enumerate}{3}
\setlist[UR]{label=($\mathcal{C}_{\arabic*})$}
\begin{UR}
\begin{multicols}{2}
\item\label{c2} $\g{n}\rightarrow \infty$;\\
\item\label{c3} $\s{n}\rightarrow 0;$
\end{multicols}
\item \label{c1} There exist $n_0, m_0$ such that  $\g{n+m}>\g{n}$ holds for  all $n\geq n_0$ and all $m\geq m_0$.\\~\\
A sequence $\gamma_n$ which satisfies condition \ref{c1} will be called an 
\emph{almost increasing sequence}; an \emph{almost decreasing sequence} is defined in an analogous way. Clearly, every increasing sequence is also an almost increasing sequence with $n_0=0$ and $m_0=1$.
\begin{multicols}{2}
\item\label{c7} $\displaystyle{\sum_{j=0}^{\infty}\frac{1}{\gamma_j}}=\infty$; 
\item\label{c8} there exists  $\kappa>1$ such that $\ \displaystyle{\sum_{j=0}^{\infty}\frac{1}{\gamma_j^\kappa}}<\infty$;
\end{multicols}
\begin{multicols}{2}
    \item\label{c5} $\displaystyle{\sum_{n=0}^{\infty}\frac{|\s{n}|}{\gamma_n^2}<\infty;}$
    \item\label{c6} $\displaystyle{\sum_{n=0}^{\infty}\frac{|\ds{n}|}{\g{n}}<\infty.}$
    \end{multicols}
\end{UR} 

Note that if the Hermite polynomials are normalised into a corresponding orthonormal family with respect to the weight $W(x)=\e^{-x^2}/\sqrt{\pi}$, then their recursion coefficients are of the form 
$\g{n}=\left(n+1\right)^{{1}/{2}}/\sqrt{2}$.

\begin{lemma}\label{lemma0} Conditions \ref{c2}-\ref{c6} are satisfied  by the Hermite polynomials, and more generally,  
\begin{enumerate}
\item[(a)] by  families with recursion coefficients of the form $\g{n}=c(n+1)^p$ for any $0<p<1$ and $c>0$;
\item[(b)] by families orthonormal with respect to the exponential weight $W(\omega)=\exp(-c|\omega|^\beta)$ for $\beta>1$ and $c>0$.
\end{enumerate}
\end{lemma}

\begin{proof}\emph{(a)}\ \  If $p>0$ then $\g{n}$ are increasing and $\g{n}\rightarrow \infty$; moreover, since for $\g{n}=c(n+1)^p$ all forward finite differences $\Delta^k(n)$ satisfy  $\Delta^k(n)=\OO{n^{p-k}}$, we obtain 
\begin{align*}  
\sum_{n=0}^{\infty}\frac{\s{n}}{\gamma_n^2}&=
\sum_{n=0}^{\infty}\OO{\frac{n^{p-1}}{n^{2p}}}=\OO{\sum_{n=0}^{\infty}n^{-p-1}}<\infty.
\end{align*}
On the other hand, if $p<1$ then  $\s{n}=\OO{n^{p-1}}\rightarrow 0$ and \ref{c7} holds; also, 
\begin{align}  
\sum_{n=0}^{\infty}{\ds{n}}&=\OO{\sum_{n=0}^{\infty}n^{p-2}}<\infty.\label{con15}
\end{align}
Note that \eqref{con15} is stronger than what is required by condition \ref{c6}.
Finally,  if $0<p<1$ then \ref{c8} holds for every $\kappa>1/p$. \\

\emph{(b)}\  Theorem 1.3 in \cite{DKMVZ} implies that for such a weight and for $\beta>1$ the recurrence coefficients satisfy 
\begin{equation*}
\frac{\gamma_n}{(n+1)^{1/\beta}}=\frac{1}{2}+\OO{(n+1)^{-\beta}}+\OO{(n+1)^{1-2\beta}}+\OO{(n+1)^{-2}}.
\end{equation*}
This in turn is easily seen to imply 
\begin{align*}\begin{split}
\frac{\s{n}}{\g{n}^2}=\OO{(n+1)^{-1-\frac{1}{\beta}}}+\OO{(n+1)^{-\beta-\frac{1}{\beta}}}+\OO{(n+1)^{1-2\beta-\frac{1}{\beta}}}+\\\OO{(n+1)^{-2-\frac{1}{\beta}}};
\end{split}
\end{align*}
\begin{align*}
\frac{\ds{n}}{\g{n}}&=\OO{(n+1)^{-\beta}}+\OO{(n+1)^{1-2\beta}}+\OO{(n+1)^{-2}}.
\end{align*}
The above three equations, together with the fact that $\beta>1$, imply \emph{(b)}.

\end{proof}
The goal of this paper is to prove the following theorem and obtain, as its consequences,  the two corollaries below, as well as  Theorem~\ref{HA}.\footnote{Theorem~\ref{HA} proves our conjecture from \cite{I1} under some additional assumptions.}

\begin{theorem}\label{t1}
Assume that the recursion coefficients $\g{n}$ which correspond to a symmetric positive definite family of orthonormal polynomials $(p_n:n\in\Nset)$ satisfy conditions \ref{c2}-\ref{c6}; then the limit $\lim_{n\rightarrow \infty}\g{n}(p^2_{n}(\omega)+p^2_{n+1}(\omega))$ exists for every $\omega\in\Rset$; moreover, for every $B>0$ there exist $m_B$ and $M_B$ such that $0<m_B<M_B<\infty$ and such that  for all $\omega$ which satisfy $|\omega|\leq B$,
\begin{equation}\label{eqbounds}
m_B\leq\lim_{n\rightarrow \infty}\g{n}(p^2_{n}(\omega)+p^2_{n+1}(\omega)) \leq M_B
\end{equation}
with the limit converging uniformly on the set of all $\omega$ such that $|\omega|\leq B$.
\end{theorem}
Note that this is in contrast with the case when the recursion coefficients are bounded; a classical result 
of Nevai implies that in the bounded case the sequence $\g{n}(p^2_{n}(\omega)+p^2_{n+1}(\omega))$ cannot converge; see \cite{PN}, page 140, formula (16). 
So it is not surprising that in our case, the slower the $\g{n}$ grow to infinity, the slower the above limit converges.

\begin{corollary}\label{corollary1}
Let $\gamma_n$ be as in Theorem~\ref{t1}; then the limits below exist and satisfy
\begin{equation}\label{eqsum1}
\lim_{n\rightarrow \infty}
\frac{\sum_{k=0}^np^2_k(\omega)}{\sum_{k=0}^n\frac{1}{\gamma_k}}=\frac{1}{2}\lim_{n\rightarrow \infty}\g{n}(p^2_{n}(\omega)+p^2_{n+1}(\omega))
\end{equation}
and convergence of the two limits is uniform on every compact set.
\end{corollary}

\begin{corollary}\label{corollary2} If 
 $\gamma_n=c(n+1)^p$ for some $c>0$ and some $p$ such that $0<p<1$, then 
 \[
\displaystyle{
0<\lim_{n\rightarrow\infty}\frac{\sum_{k=0}^np^2_k(\omega)}{(n+1)^{1-p}}<\infty,
}
\]
and convergence of the limit  is uniform on every compact set.
\end{corollary}


\section{A Representation of Orthogonal Polynomials}

To estimate asymptotic behaviour of the sum $p_{m}^2(\omega)+p_{m+1}^2(\omega)$  as $m\rightarrow\infty$, we will consider functions of the form ${\ii}^m p_{m}(\omega)+{\ii}^{m+1}p_{m+1}(\omega)$; in this way we have $|{\ii}^m p_{m}(\omega)+{\ii}^{m+1}p_{m+1}(\omega)|^2=p_{m}^2(\omega)+p_{m+1}^2(\omega)$.  What the real and what the imaginary part of such a function is  depends on the parity of $m$.  For simplicity, we will assume that $m=2n$ and  define
\begin{align}\label{lab4}
E_n(\omega)&={\ii}^{2n}p_{2n}(\omega)+{\ii}^{2n+1}p_{2n+1}(\omega)=(-1)^n(p_{2n}(\omega)+{\ii}\,p_{2n+1}(\omega));
\end{align}
thus, 
\begin{align}\label{absE}
p_{2n}^2(\omega)+p_{2n+1}^2(\omega)=|E_n(\omega)|^2.
\end{align}
For odd $m$ we set $E_n(\omega)={\ii}^{2n+1}p _{2n+1}(\omega)+{\ii}^{2n+2}p_{2n+2}(\omega)$, and 
all of our  arguments go through with minor changes only.\footnote{These changes can be found in a \emph{Mathematica} file available online at \url{http://www.cse.unsw.edu.au/~ignjat/diff/OP2.zip}.} Also, since by our assumption \ref{c3} we have $\s{n}=\g{n+1}-\g{n}\rightarrow 0$, it follows that  $\g{n}/\g{n+1}\rightarrow 1$ and thus 
\[\LIM{n}\left(\frac{1}{\g{2n}}+\frac{1}{\g{2n+1}}\right)\frac{\g{2n}}{2}=1. \]
This implies that one of the two limits in \eqref{even} below  exists just in case the other also exists, in which case
\begin{equation}\label{even}
\frac{1}{2}\LIM{n}\g{2n}(p_{2n}^2(\omega)+p_{2n+1}^2(\omega))=\LIM{n}\frac{p_{2n}^2(\omega)+p_{2n+1}^2(\omega)}{\frac{1}{\g{2n}}+\frac{1}{\g{2n+1}}}.
\end{equation}

Thus, instead of proving \eqref{eqbounds} we will prove that the following limit exists and satisfies
\begin{equation}\label{7}
0<\LIM{n}\frac{p^2_{2n}(\omega)+p^2_{2n+1}(\omega)}{\frac{1}{\g{2n}}+\frac{1}{\g{2n+1}}}<\infty,
\end{equation}
and instead of proving \eqref{eqsum1} we will prove the following, more symmetric equivalent form of it,
\begin{equation}\label{8}
{\lim_{n\rightarrow \infty}
\frac{\sum_{k=0}^np^2_k(\omega)}{\sum_{k=0}^n\frac{1}{\gamma_k}}=\lim_{n\rightarrow \infty}{\frac{p^2_{2n}(\omega)+p^2_{2n+1}(\omega)}{\frac{1}{\g{2n}}+\frac{1}{\g{2n+1}}}}},
\end{equation}
as well as that the convergence of both limits is uniform on every compact interval.\\

We now look for a recurrence satisfied by $E_n(\omega)$ given by \eqref{lab4}. Substituting $n$ by $2n-1$ in the  three term recurrence \eqref{poly} and dividing both sides of the resulting equation by $\g{2n-1}$ we obtain
\begin{equation}\label{lab20}
p_{2n}(\omega)=
\frac{\omega}{\g{2n-1}}\, p_{2n-1}(\omega)-\frac{\g{2n-2}}{\g{2n-1}}\,p_{2n-2}(\omega).
\end{equation}
Multiplying both sides by ${\ii}^{2n}$ produces
\begin{equation}\label{lab2}
{\ii}^{2n}p_{2n}(\omega)=
\frac{{\ii}\,\omega}{\g{2n-1}}\, {\ii}^{2n-1}p_{2n-1}(\omega)+\frac{\g{2n-2}}{\g{2n-1}}\, {\ii}^{2n-2}p_{2n-2}(\omega).
\end{equation}
Similarly, substituting in  \eqref{poly} $n$ by $2n$ and using \eqref{lab20} to eliminate $p_{2n}(\omega)$ we obtain
\begin{align*}
\g{2n}p_{2n+1}(\omega)&=\omega\left(\frac{\omega}{\g{2n-1}}p_{2n-1}(\omega)-\frac{\g{2n-2}}{\g{2n-1}}p_{2n-2}(\omega)\right)-\g{2n-1}p_{2n-1}(\omega)\\
&=\left(\frac{\omega^2}{\g{2n-1}}-\g{2n-1}\right)p_{2n-1}(\omega)-
\frac{\omega\,\g{2n-2}}{\g{2n-1}}p_{2n-2}(\omega).
\end{align*}
Multiplying both sides by ${\ii}^{2n+1}/\g{2n}$ we obtain
\begin{equation}\label{lab3}\begin{split}
{\ii}^{2n+1}p_{2n+1}(\omega)=\left(-\frac{\omega^{2}}{\g{2n-1}\g{2n}}+\frac{\g{2n-1}}{\g{2n}}\right){\ii}^{2n-1}p_{2n-1}(\omega)+\\
\frac{{\ii}\,\omega\,\g{2n-2}}{\g{2n}\g{2n-1}}\, {\ii}^{2n-2}p_{2n-2}(\omega).\end{split}
\end{equation}

If we add  equations \eqref{lab2} and \eqref{lab3} together we obtain
\begin{equation}\label{lab6}\begin{split}
E_{n}(\omega)=\left(-\frac{\omega^{2}}{\g{2n-1}\g{2n}}+\frac{\g{2n-1}}{\g{2n}}+{\ii}\,\frac{\omega}{\g{2n-1}}\right) {\ii}^{2n-1}p_{2n-1}(\omega)+\\
\left(\frac{\g{2n-2}}{\g{2n-1}}+{\ii}\,\frac{\omega\,\g{2n-2}}{\g{2n-1}\g{2n}}\right)\,{\ii}^{2n-2}p_{2n-2}(\omega).\end{split}
\end{equation}

Since from \eqref{lab4} we have
\begin{align}\label{PT}
&{\ii}^{2n-2}p_{2n-2}(\omega)=\frac{E_{n-1}(\omega)+\overline{E_{n-1}(\omega)}}{2}; &&{\ii}^{2n-1}p_{2n-1}(\omega)=\frac{E_{n-1}(\omega)-\overline{E_{n-1}(\omega)}}{2},&
\end{align}
after corresponding substitutions of \eqref{PT} in \eqref{lab6} we obtain 
\begin{align}\label{EE}
&E_{n}(\omega)=\nonumber\\
&\left(-\frac{\omega^{2}}{2\g{2n-1}\g{2n}}+\frac{\g{2n-1}}{2\g{2n}}+\frac{\g{2n-2}}{2\g{2n-1}}+{\ii}\left(\frac{\omega}{2\g{2n-1}}+\frac{\omega\,\g{2n-2}}{2\g{2n-1}\g{2n}}\right)\right)E_{n-1}(\omega)+\nonumber\\
&\left(\frac{\omega^{2}}{2\g{2n-1}\g{2n}}-\frac{\g{2n-1}}{2\g{2n}}+\frac{\g{2n-2}}{2\g{2n-1}}+{\ii}\left(-\frac{\omega}{2\g{2n-1}}+\frac{\omega\,\g{2n-2}}{2\g{2n-1}\g{2n}}\right)\right)\overline{E_{n-1}(\omega)}.
\end{align}

Since  the families of orthonormal polynomials considered in this paper are symmetric, we will restrict our attention to $\omega>0$; in all of our propositions the case when $\omega=0$ can easily be  handled  separately. Moreover, we will assume that $\omega>0$ is fixed and, to make our formulas more readable,  we will sometimes suppress $\omega$ in our notation; thus, for example, we will write $E_n$ instead of $E_n(\omega)$.\\

To get a more compact form of equality \eqref{EE} we define for all $n\geq 1$,
\begin{align}
\zt{n}&=-\frac{\omega^{2}}{2\g{2n-1}\g{2n}}+\frac{\g{2n-1}}{2\g{2n}}+\frac{\g{2n-2}}{2\g{2n-1}}+\ii\left(
\frac{\omega}{2\g{2n-1}}+\frac{\omega\,\g{2n-2}}{2\g{2n-1}\g{2n}}\right);\label{basic1}\\
\ztt{n}&=\frac{\omega^{2}}{2\g{2n-1}\g{2n}}-\frac{\g{2n-1}}{2\g{2n}}+\frac{\g{2n-2}}{2\g{2n-1}}+\ii\left(-\frac{\omega}{2\g{2n-1}}+\frac{\omega\,\g{2n-2}}{2\g{2n-1}\g{2n}}\right).\label{basic2}
\end{align}
Equation~\eqref{EE} now becomes
\begin{align}\label{main1}
E_n&=\zt{n}E_{n-1}+\ztt{n}\overline{E_{n-1}};
\end{align}
thus, while polynomials $p_n(\omega)$ satisfy a three term recurrence, $E_n(\omega)$ satisfy a recurrence with only two terms.\\

Let  $\Te{-1}=0$ and
for all $n\geq 0$  let $\Te{n}$ be the least number larger than $\Te{n-1}$ such that 
\[\Te{n}\equiv \arg E_n\mod 2\pi.\]
Thus,  for $n\geq 0$, $\Te{n}$ is a sequence of positive reals, monotonically increasing in $n$, such that
\begin{align}\label{polar}
&E_n=|E_n|{\e}^{\ii \Te{n}};&&p_{2n}(\omega)=(-1)^n|E_n|\cos\Te{n}; && p_{2n+1}(\omega)=(-1)^n|E_n|\sin\Te{n}.&
\end{align} 

In signal processing terminology,  $\Te{n}$ is the \emph{unwound phase of} $E_n$. We now define 
\begin{equation}\label{delta1}
\del{n}=\Te{n}-\Te{n-1}>0.
\end{equation}

By dividing both sides of  \eqref{main1} by $E_{n-1}$ we obtain
\begin{align}\label{new1}
\frac{E_n}{E_{n-1}}=\zt{n}+\ztt{n}\frac{\overline{E_{n-1}}}{E_{n-1}}=\zt{n}+\ztt{n}{\e}^{-2\ii \Te{n-1}}.
\end{align}
Since
\begin{align*}
\frac{E_n}{E_{n-1}}=\frac{|E_n|}{|E_{n-1}|}\,{\e}^{\ii (\Te{n}-\Te{n-1})}=\frac{|E_n|}{|E_{n-1}|}\,{\e}^{\ii \Delta_{n}},
\end{align*}
from \eqref{new1} we obtain 
\begin{align}\label{main}
\frac{|E_n|}{|E_{n-1}|}\,{\e}^{\ii \Delta_{n}}=
\zt{n}+\ztt{n}{\e}^{-2\ii \Te{n-1}},
\end{align}
which implies
\begin{align}\label{ee3}
\frac{|E_n|}{|E_{n-1}|}&=
|\zt{n}+\ztt{n}{\e}^{-2\ii \Te{n-1}}|.
\end{align}
Let us define
\begin{align}\label{mult}
&\mult{0}=|E_0|;&&
\mult{n}=|\zt{n}+\ztt{n}{\e}^{-2\ii \Te{n-1}}|,\ \ (n>0);&
\end{align}
then \eqref{ee3} is equivalent to
\begin{align}\label{AA}
|E_n|&=|E_{n-1}|\mult{n}.
\end{align}
Consequently,
\begin{align*}
|E_n|&=\prod_{j=0}^{n}\mult{j}
~ \label{ee4}
\end{align*}
and this and \eqref{absE} imply
\begin{align*}
\frac{p_{2n}(\omega)^2+p_{2n+1}(\omega)^2}{\frac{1}{\g{2n}}+\frac{1}{\g{2n+1}}}&
=\frac{\prod_{j=0}^{n}\mult{j}^2}{\frac{1}{\g{2n}}+\frac{1}{\g{2n+1}}}.
\end{align*}
Taking the logarithm of both sides and letting 
\begin{equation}\label{largelog}
\mathcal{S}_n=2\sum_{j=0}^{n}\ln\mult{j}-\ln\left(\frac{1}{\g{2n}}+\frac{1}{\g{2n+1}}\right),
\end{equation}
we conclude that, in order  to prove Theorem~\ref{t1}, it is enough to prove that 
$\mathcal{S}_n(\omega)$ converges to a finite limit as $n\rightarrow\infty$, uniformly in $\omega$ from a compact set. \\

Let us define 
\begin{align*}
\lambda_0&=\frac{1}{\frac{1}{\g{0}}+\frac{1}{\g{1}}};&
\lambda_n&=\frac{\frac{1}{\g{2n-2}}+\frac{1}{\g{2n-1}}}{\frac{1}{\g{2n}}+\frac{1}{\g{2n+1}}},\ \ (n\geq 1).&
\end{align*} 
We can now represent $-\ln\left(\frac{1}{\g{2n}}+\frac{1}{\g{2n+1}}\right)$ as a telescopic sum,
\begin{align}\label{telescope}
-\ln\left(\frac{1}{\g{2n}}+\frac{1}{\g{2n+1}}\right)=\sum_{j=1}^{n+1}\ln\lambda_{j-1}
\end{align}
and obtain from \eqref{largelog}
\begin{align}
\mathcal{S}_n&=2\ln \mult{0}+\ln\lambda_{n}+\sum_{j=1}^{n}\left(2\ln\mult{j}+\ln\lambda_{j-1}\right)\nonumber\\
&=\ln \mult{0}+\ln\mult{n}+\ln\lambda_{n}+\sum_{j=1}^{n}\left(\ln\mult{j-1}+\ln\mult{j}+\ln\lambda_{j-1}\right).
\label{SS}
\end{align}
The reasons for introducing the telescopic sum \eqref{telescope} and for pairing $\ln\mult{j-1}$ with $\ln\mult{j}$ in \eqref{SS} will be clear later.\footnote{See the comment at the end of the proof of Lemma~\ref{FG1}, footnote \ref{ftn1} and the comment after equation \eqref{normw}.}
Before proceeding with the proof of convergence of $\mathcal{S}_n$, we must first prove some elementary properties of the basic sequences $\zt{n}, \ztt{n}$ and $\del{n}$.


\section{Properties of the Basic Sequences}

Let us define 
\begin{equation}\label{epsdef}
\epsilon_n=\frac{\g{2n-2}\g{2n}-\g{2n-1}^2}{\g{2n-1}}.
\end{equation} 

\begin{lemma}\label{l0}
\begin{align}\label{epsilon}
\epsilon_n=\ds{2n-2}-\frac{\s{2n-2}\s{2n-1}}{\g{2n-1}}\rightarrow 0.
\end{align}
\end{lemma}
\begin{proof} Lemma  \ref{l0} and Lemma \ref{l1}  both follow by straightforward computations from definitions of $a_n$, $b_n$ and $\epsilon_n$ and condition \ref{c3}.\footnote{Computations which are not presented in every detail in this paper have been verified using the symbolic functionality of  the \emph{Mathematica}{\texttrademark} software. All asymptotic representations  were also checked numerically, using ten families of orthonormal polynomials. \emph{Mathematica} files containing these symbolic verifications as well as files containing numerical tests are available online at \url{http://www.cse.unsw.edu.au/~ignjat/diff/OP.zip}. Expanding the compressed file \texttt{OP.zip} produces a folder \texttt{OP} with two subfolders, \texttt{OP/symbolic/}  and \texttt{OP/numerical/} containing files which can be viewed using either \emph{Mathematica} software package or \emph{Wolfram CDF Player} available free of charge at  \url{http://www.wolfram.com/cdf-player/ }. In particular, \emph{Mathematica} verifications of Lemma \ref{l0} and  Lemma  \ref{l1}  are in file \texttt{OP/symbolic/0\_lemmas\_\ref{l0}\_and\_\ref{l1}.nb}.} 
\end{proof}

\begin{lemma}\label{l1}
\begin{align} \label{cs}\begin{split}
&(a)\hspace*{2mm}\Re(\zt{n})=\frac{\g{2n-1}}{\g{2n}}\left(1-\frac{\omega^2}{2\,\g{2n-1}^2}+\frac{\epsilon_{n}}{2\,\g{2n-1}}\right);\\
&(b)\hspace*{2mm}\Im(\zt{n})=\frac{\omega}{\g{2n}}\left(1+\frac{\ds{2n-2}}{2\,\g{2n-1}}\right);
\end{split}
\end{align}
\begin{align} \begin{split}
&(a)\hspace*{2mm}\Re(\ztt{n})=\frac{\omega}{2\g{2n}}\left(\frac{\omega}{\,\g{2n-1}}+\frac{\epsilon_n}{\omega}\right);\\
&(b)\hspace*{2mm}\Im(\ztt{n})=-\frac{\omega\left(\s{2n-1}+\s{2n-2}\right)}{2\,\g{2n-1}\g{2n}}.\label{sn}
\end{split}
\end{align}\hfill\qed
\end{lemma}

\begin{corollary}\label{corarg}$\arg\zt{n}\rightarrow 0$ and $\arg\zt{n}>0$ for all sufficiently large $n$.
\end{corollary}
\begin{proof}
From \eqref{cs} we have 
\begin{align}\label{te}
\arg\zt{n}&=\arctan\frac{\frac{\omega}{\g{2n-1}}\left(1+\frac{\ds{2n-2}}{2\g{2n-1}}\right)}{1-\frac{\omega^2}{2\g{2n-1}^2}+\frac{\epsilon_n}{2\g{2n-1}}}=\frac{\omega}{\g{2n-1}}+\OO{\frac{1}{\g{2n-1}^2}},
\end{align}
which implies both claims.
\end{proof}

\begin{corollary}\label{l2}
\begin{align}\label{asts}
&(a)\ \ |\ztt{n}|\rightarrow 0;&&(b)\ \ |\zt{n}|\rightarrow 1. &
\end{align}
\end{corollary}
\begin{proof}
From \eqref{sn}$(a)$ and \eqref{sn}$(b)$ we have
\begin{align}\label{aa1}
\amt{n}^2=\cst{n}^2+\snt{n}^2=\frac{\omega^2}{4\g{2n}^2}\left(\left(\frac{\omega}{\g{2n-1}}+\frac{\epsilon_n}{\omega}\right)^2+\left(\frac{\s{2n-1}+\s{2n-2}}{\g{2n-1}}\right)^2\right),
\end{align}
which, together with Lemma \ref{l0} and condition \ref{c3},  implies the $(a)$ part. On the other hand, directly from definitions \eqref{basic1} and \eqref{basic2}, after some simplifications  we have 
\begin{align*}
\cs{n}^2-\cst{n}^2&=\frac{\g{2n-2}}{\g{2n-1}}\left(-\frac{\omega^2}{\g{2n-1}\g{2n}}+\frac{\g{2n-1}}{\g{2n}}\right);\\
\sn{n}^2-\snt{n}^2&=\frac{\omega^2\g{2n-2}}{\g{2n-1}^2\g{2n}}.
\end{align*}
Summing these two equations produces 
\begin{align}\label{bb1}
\am{n}^2-\amt{n}^2&=\frac{\g{2n-2}}{\g{2n}}=1-\frac{\s{2n-2}+\s{2n-1}}{\g{2n}},
\end{align}
which, together with $(a)$,  proves the $(b)$ part of the corollary.
\end{proof}

\begin{corollary}\label{l4}
 $\Im(\zt{n})>|\ztt{n}|$ for all sufficiently large $n$.
\end{corollary}
\begin{proof}
We first obtain from \eqref{cs}(b)
\begin{align*}
&\sn{n}=\\
&\ \ \frac{\omega}{\g{2n-1}}\frac{\g{2n-1}}{\g{2n}}\left(1+\frac{\s{2n-1}-\s{2n-2}}{2\g{2n-1}}\right)=\frac{\omega}{\g{2n-1}}
\left(\frac{\g{2n-1}}{\g{2n}}+\frac{\s{2n-1}-\s{2n-2}}{2\g{2n}}\right)=\\
&\ \ \frac{\omega}{\g{2n-1}}
\left(1+\frac{\g{2n-1}-\g{2n}}{\g{2n}}+\frac{\s{2n-1}-\s{2n-2}}{2\g{2n}}\right)=\frac{\omega}{\g{2n-1}}
\left(1-\frac{\s{2n-1}+\s{2n-2}}{2\g{2n}}\right).
\end{align*}
This, together with \eqref{aa1}, yields
\begin{align*}
&\sn{n}^2-\amt{n}^2=\\
&\frac{\omega^2}{\g{2n-1}^2}\left(\left(1-\frac{\s{2n-1}+\s{2n-2}}{2\g{2n}}\right)^2-\left(\frac{\omega}{\g{2n-1}}+\frac{\epsilon_n}{\omega}\right)^2\frac{\g{2n-1}^2}{4\g{2n}^2}-\left(\frac{\s{2n-1}+\s{2n-2}}{2\g{2n}}\right)^2\right)\\
&=\frac{\omega^2}{\g{2n-1}^2}\left(1-\frac{\s{2n-1}+\s{2n-2}}{\g{2n}}-\frac{1}{4}\left(\frac{\omega}{\g{2n-1}}+\frac{\epsilon_n}{\omega}\right)^2\left(1-\frac{\s{2n-1}}{\g{2n}}\right)^2 \right).
\end{align*}
Since equations \eqref{cs} together with Lemma \ref{l0} and Condition \ref{c3} also imply that eventually  $\Im(\zt{n})>0$,
we obtain the claim of the corollary.
\end{proof}

We now return to equation \eqref{main}. Since equations \eqref{cs} together with Lemma \ref{l0} and Condition \ref{c3} imply that eventually $\Re(\zt{n})>0$ and $\Im(\zt{n})>0$, we obtain  $0<\arg{\zt{n}}<\pi/2$. On the other hand, by Corollary \ref{l4}, $\Im(\zt{n})>|\ztt{n}|$ for all sufficiently large $n$. Thus, 
$\arg(\zt{n}+\ztt{n}{\e}^{-2\ii \Te{n-1}})>0$; see Figure~\ref{small-plot}.

\begin{figure}[h!]
\begin{center}
\includegraphics[width=4in]{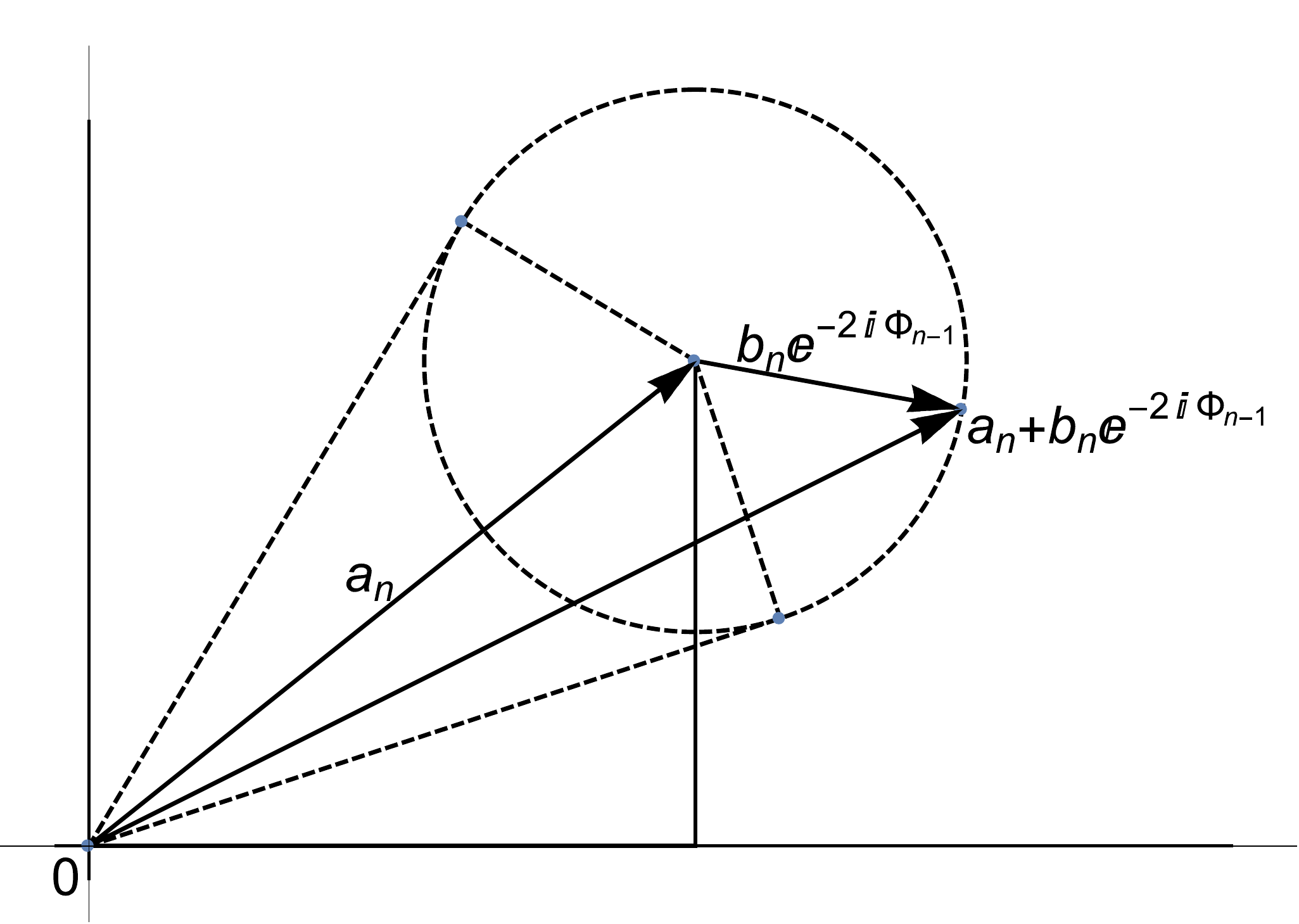}
\caption{{\bf }}
\label{small-plot}
\end{center}
\end{figure}

Since  \eqref{main} implies that $\del{n}-\arg\bigl(\zt{n} + \ztt{n}{\e}^{-2\ii \Te{n-1}}\bigr)=2k\pi$ for some $|k|\leq 1$ and since $2\pi\geq\del{n}>0$ and $\pi\geq\arg(\zt{n}+\ztt{n}{\e}^{-2\ii \Te{n-1}})>0$, we obtain that for all sufficiently large $n$,
\begin{align}\label{td}
\del{n}&=\arg\bigl(\zt{n} + \ztt{n}{\e}^{-2\ii \Te{n-1}}\bigr).
\end{align}

Note that also  (see again Figure \ref{small-plot})
\begin{align}\label{bt}
\arg(\zt{n})-\arcsin\frac{|\ztt{n}|}{|\zt{n}|}\leq\arg(\zt{n}+\ztt{n}{\e}^{-2\ii \Te{n-1}})\leq\arg(\zt{n})+\arcsin\frac{|\ztt{n}|}{|\zt{n}|}.
\end{align}
On the other hand, \eqref{aa1},  \eqref{bb1}  and
a series expansion of $\arcsin\sqrt{x}$ imply that\footnote{\emph{Mathematica} verification of equation \eqref{arcsin} is in file \texttt{OP/symbolic/1\_eq\_\ref{arcsin}.nb}.}
\begin{equation}\label{arcsin}\arcsin\frac{|\ztt{n}|}{|\zt{n}|}=\OO{\frac{|\epsilon_n|}{\g{2n-1}}}+\OO{\frac{1}{\g{2n-1}^2}},\end{equation}
which, together with \eqref{te}, \eqref{bt} and Lemma \ref{l0} yield
\begin{align}\label{dlt}
 \del{n}=\frac{\omega}{\g{2n-1}}(1 + o(1)).
\end{align}
Consequently, Condition \ref{c7} implies
\begin{equation}\label{Te}
\Te{n}=\Te{0}+\sum_{k=1}^n\Delta_k \sim\sum_{k=1}^n\frac{\omega}{\g{2k-1}}\rightarrow\infty.
\end{equation}~\\

Equations \eqref{mult} and \eqref{td} express $\mult{n}$ and $\del{n}$ via $\Te{n-1}$.  For a reason which will be clear later\footnote{See footnote \ref{ftn1}.}, we need to represent $\mult{n-1}$ and $\Delta_{n-1}$ also via $\Te{n-1}$, rather than $\Te{n-2}$. 
To this end, taking the complex conjugate of both sides of equation~\eqref{main1}, we obtain 
\begin{align}\label{mainc}
\overline{E_n}&=\overline{\zt{n}}\;\overline{E_{n-1}}+\overline{\ztt{n}}{E_{n-1}}.
\end{align}
Multiplying both sides of \eqref{main1} by $\overline{a_n}$, multiplying both sides of \eqref{mainc} by $b_n$ and subtracting the corresponding sides of thus obtained two equations produces 
\begin{align}\label{ll1}
E_{n-1}&=\frac{\overline{\zt{n}}E_{n}-\ztt{n}\overline{E_{n}}}{|{\zt{n}}|^2-|{\ztt{n}}|^2},
\end{align}
which implies 
\begin{align*}
\frac{E_{n-1}}{E_n}&=\frac{\overline{\zt{n}}-\ztt{n}{\e}^{-\ii 2\Te{n}}}{|{\zt{n}}|^2-|{\ztt{n}}|^2}.
\end{align*}
Thus,  
\begin{align*}
\frac{|E_{n-1}|}{|E_n|}{\e}^{-\ii\del{n}}&=\frac{\overline{\zt{n}}-\ztt{n}{\e}^{-\ii 2\Te{n}}}{|{\zt{n}}|^2-|{\ztt{n}}|^2}.
\end{align*}
Taking the complex conjugates of both sides and using \eqref{AA} we obtain
\begin{align}\label{nnm}
\frac{{\e}^{\ii\del{n}}}{\mult{n}}&=\frac{{\zt{n}}-\overline{\ztt{n}}{\e}^{\ii 2\Te{n}}}{|{\zt{n}}|^2-|{\ztt{n}}|^2}.
\end{align}
This implies 
\begin{align}\label{lll1}
{\mult{n}}&=\frac{|{\zt{n}}|^2-|{\ztt{n}}|^2}{|{\zt{n}}-\overline{\ztt{n}}{\e}^{\ii 2\Te{n}}|},
\end{align}
and, using the same reasoning as in the derivation of \eqref{td}, 
\begin{align}\label{ll5}
\del{n}&=\arg\left({\zt{n}}-\overline{\ztt{n}}{\e}^{\ii 2\Te{n}}\right).
\end{align}
Finally, substituting $n$ with $n-1$  in \eqref{lll1} and \eqref{ll5} we get
\begin{align}
{\mult{n-1}}&=\frac{|{\zt{n-1}}|^2-|{\ztt{n-1}}|^2}{|{\zt{n-1}}-\overline{\ztt{n-1}}{\e}^{\ii 2\Te{n-1}}|};\label{mun-1}\\
\del{n-1}&=\arg\left({\zt{n-1}}-\overline{\ztt{n-1}}{\e}^{\ii 2\Te{n-1}}\right).\label{deln-1}
\end{align}


\section{Representing $\mathcal{S}_n$ as a Riemann Sum}

A part of our strategy for proving convergence of $\mathcal{S}_n$ is to represent $\mathcal{S}_n$ as a Riemann sum.
Using  \eqref{mult} and  \eqref{mun-1} we get
\begin{align}\label{multmult1}\begin{split}
\ln\mult{n}+\ln\mult{n-1}+\ln\lambda_{n-1}=\ln({|\zt{n-1}}|^2-|{\ztt{n-1}}|^2)+\ln\lambda_{n-1}\\+\ln\left|{\zt{n} }+\ztt{n}{\e}^{-2\ii \Te{n-1}}\right|-\ln\left|{{\zt{n-1}}-\overline{\ztt{n-1}}{\e}^{2\ii\Te{n-1}}}\right|.\end{split}
\end{align}

Similarly, from  \eqref{td} and \eqref{deln-1} we also get that
\begin{align}\begin{split}
\del{n-1}+\del{n}=\arg\bigl(\zt{n-1}-\overline{\ztt{n-1}}{\e}^{2\ii\Te{n-1}}\bigr)+
\arg\bigl(\zt{n} + {\ztt{n}}{\e}^{-2\ii \Te{n-1}}\bigr).
\label{deldel1}\end{split}
\end{align}
Note that \eqref{dlt} implies that both summands in \eqref{deldel1} are positive and converge to zero; thus, since for all $z\in\Cset$ which are outside the branch cut $(-\infty,0]$ of the logarithm function we have
\begin{align*}
&\ln |z|= \frac{1}{2}\bigl(\ln\overline{z}+\ln z\bigr);&&
\arg z=\frac{\ii}{2}\bigl(\ln \overline{z}-\ln z\bigr),&
\end{align*}
equations \eqref{multmult1} and \eqref{deldel1} can be transformed into
\begin{align}\label{mult22}
\begin{split}
\ln\mult{n}+\ln\mult{n-1}+\ln\lambda_{n-1}=\ln({|\zt{n-1}}|^2-|{\ztt{n-1}}|^2)+\ln\lambda_{n-1}+\\\frac{1}{2}\ln\overline{\zt{n} +\ztt{n}{\e}^{-2\ii \Te{n-1}}}+
\frac{1}{2}\ln({\zt{n} }+\ztt{n}{\e}^{-2\ii \Te{n-1}})-\\\frac{1}{2}\ln\overline{\zt{n-1}-\overline{\ztt{n-1}}{\e}^{2\ii\Te{n-1}}}-\frac{1}{2}\ln({{\zt{n-1}}-\overline{\ztt{n-1}}{\e}^{2\ii\Te{n-1}}});
\end{split}
\end{align}
\begin{align}\label{del22}
\begin{split}
\del{n-1}+\del{n}=\frac{\ii}{2}\ln\overline{\zt{n} +\ztt{n}{\e}^{-2\ii \Te{n-1}}}-\frac{\ii}{2}\ln({\zt{n} }+\ztt{n}{\e}^{-2\ii \Te{n-1}})\\+\frac{\ii}{2}\ln\overline{\zt{n-1}-\overline{\ztt{n-1}}{\e}^{2\ii\Te{n-1}}}-
\frac{\ii}{2}\ln(\zt{n-1}-\overline{\ztt{n-1}}{\e}^{2\ii\Te{n-1}}).
\end{split}
\end{align}
~\\
Thus, quite remarkably and very fortunately for our proof,  the summands in the sum  appearing in \eqref{SS}, i.e.,  $ \ln\mult{j}+\ln\mult{j-1}$ as well as  $\del{j-1}+\del{j}$ are both expressible via the same logarithms, $\ln({\zt{j} }+\ztt{j}{\e}^{-2\ii \Te{j-1}})$ and $\ln(\zt{j-1}-\overline{\ztt{j-1}}{\e}^{2\ii\Te{j-1}})$. This fact will be used for representing $\mathcal{S}_n$ as a Riemann sum. Let us define for all $n\geq 1$,

\begin{align}\label{*1}
\begin{split}
F_n(t)=2\ln({|\zt{n-1}}|^2-|{\ztt{n-1}}|^2)+2\ln\lambda_{n-1}+\ln(\overline{\zt{n}} +\overline{\ztt{n}}{\e}^{\ii t})+\ln(\zt{n} +\ztt{n}{\e}^{-\ii t})\\-\ln(\overline{\zt{n-1}}-{\ztt{n-1}}{\e}^{-\ii t})
-\ln(\zt{n-1}-\overline{\ztt{n-1}}{\e}^{\ii t});
\end{split}
\end{align}
\begin{align}\begin{split}
\label{*2}
G_n(t)=\ii \left(\ln(\overline{\zt{n}} +\overline{\ztt{n}}{\e}^{\ii t})-\ln(\zt{n}+\ztt{n}{\e}^{-\ii t})+\ln(\overline{\zt{n-1}}-\ztt{n-1}{\e}^{-\ii t})-\right.\\
\left.\ln(\zt{n-1}-\overline{\ztt{n-1}}{\e}^{\ii t})\right).\end{split}
\end{align} 
If we let 
\begin{align*}
H_n(t)=\frac{F_n(t)}{G_n(t)},
\end{align*} 
then \eqref{mult22} and \eqref{del22} imply that for all $n\geq 2$, 
\begin{align}
F_n(2\Te{n-1})&=2(\ln\mult{n}+\ln\mult{n-1}+\ln\lambda_{n-1});\\
G_n(2\Te{n-1})&=2(\del{n-1}+\del{n}).\label{sumDel}
\end{align}
Thus, 
\begin{align}
H_n(2\Te{n-1})(\del{n-1}+\del{n})=\ln\mult{n}+\ln\mult{n-1}+\ln\lambda_{n-1},
\end{align}
and consequently
\begin{align*}
\sum_{j=1}^{n}H_j(2\Te{j-1})(\del{j-1}+\del{j})&=\sum_{j=1}^n( \ln\mult{j}+\ln\mult{j-1}+\ln\lambda_{j-1}).
\end{align*}
Combining this with \eqref{SS}  we obtain
\begin{align}
\mathcal{S}_n&=\ln \mult{0}+\ln\mult{n}+\ln\lambda_{n}+\sum_{j=1}^{n}H_j(2\Te{j-1})(\del{j-1}+\del{j}).
\end{align}
It is easy to see that condition \ref{c3} implies $\lambda_n\rightarrow 1$; see for example equation \eqref{alambda}.  Since 
\[|\zt{n}|-|\ztt{n}|\leq |\zt{n}+\ztt{n}{\e}^{-2\ii \Te{n}}|\leq |\zt{n}|+|\ztt{n}|,\]
Corollary \ref{l2} implies  $\mult{n}=|\zt{n}+\ztt{n}{\e}^{-2\ii \Te{n}}|\rightarrow 1$. Thus, to prove that $\mathcal{S}_n$ is convergent, it is enough to show that the sum  
\begin{align}\label{s*}
\mathcal{S}_n^\ast&=\sum_{j=1}^{n} H_j(2\Te{j-1})(\del{j-1}+\del{j})
\end{align}
converges to a finite limit as $n\rightarrow \infty$.\footnote{\label{ftn1}Formula \eqref{s*} explains why we paired $\del{n-1}$ with $\del {n}$ (and thus also $\mult{n-1}$ with $\mult{n}$) and why we expressed $\del{n-1}$ in terms of $\Te{n-1}$. Such pairing has other advantages as well; see, for example, the comment after equation \eqref{normw}.}

We now note that the sum  $\mathcal{S}_n^\ast$  already resembles a Riemann sum, with a partition of the interval of integration $[2\Te{1}-\del{1}, 2\Te{n-1}+\del{n}]$ into sub-intervals $[2\Te{j-1}-\del{j-1},2\Te{j-1}+\del{j}]$ and the integrand $H_j(t)$ evaluated at sampling points $2\Te{j-1}$, except that $H_j(t)$ is a sequence of functions, rather than a single function.  However, since functions $ H_j(t)$  are  $2\pi$ periodic, we can expand them into their Fourier series which, as we shall see, will reduce $\mathcal{S}^\ast_n$ to Riemann sums for integrals of some damped complex exponentials. We will then show that these integrals converge, as well as  that the errors of approximating these integrals by the corresponding Riemann sums also converge, which will entail convergence of $\mathcal{S}_n^\ast$ as well.  

Before proceeding with such a strategy, we first reduce functions $ H_j(t)$ to functions which do not contain any finite differences, plus some remainders which are absolutely summable. This is not just a simplification, but is crucial for obtaining almost monotonic Fourier coefficients, which we will need to carry out  our arguments.


\section{Asymptotic Representation of the Basic Sequences $a_n$ and $b_n$}
If $f(n,\omega,t)$ is a complex valued function of an integer $n$ and real variables $t,\omega$ and if $g(n, \omega)>0$ is a positive real valued function of $n, \omega$, then we write $f(n,\omega,t)=\OO{g(n,\omega)}$ to denote the fact that for every fixed $\omega>0$ there exists $M_{\omega}>0$ such that   $|f(n,\omega,t)|\leq M_{\omega}\, g(n,\omega)$ for all $t\in\Rset$. A direct inspection of our estimates reveals that, in general, $M_{\omega}\rightarrow\infty$ as $\omega\rightarrow\infty$; thus, in our results we can conclude that convergence is uniform only for $\omega$ belonging to a compact set. 
Also, to simplify our notation let us define
\begin{align}\label{eta}
\eta_n=|\s{2n-4}|+|\s{2n-3}|+|\s{2n-2}|+|\s{2n-1}|.
\end{align}
\begin{lemma}\label{basicn}
\begin{align}\label{csa}
\Re(\zt{n})&=1-\frac{\omega^2}{2\g{2n-1}^2}-\frac{\s{2n-1}+\s{2n-2}}{2\g{2n-1}}+\OO{\frac{|\s{2n-1}|}{\g{2n-1}^2}};\\
\Im(\zt{n})&=\frac{\omega}{\g{2n-1}}\left(1-\frac{\s{2n-1}+\s{2n-2}}{2\g{2n-1}}\right)+\OO{\frac{|\s{2n-1}|}{\g{2n-1}^3}};\label{sna}\\
\Re(\ztt{n})&=\frac{\omega^2}{2\g{2n-1}^2}+\frac{\ds{2n-2}}{2\g{2n-1}}+\OO{\frac{|\s{2n-1}|}{\g{2n-1}^2}};\label{csta}\\
\Im(\ztt{n})&=-\frac{\omega(\s{2n-1}+\s{2n-2})}{2\g{2n-1}^2}+\OO{\frac{|\s{2n-1}|}{\g{2n-1}^3}}.\label{snta}
\end{align}
\end{lemma}
\begin{proof}
Note that
\begin{align*}
{\frac{\g{2n-1}}{\g{2n}}}=1+\frac{\g{2n-1}-\g{2n}}{\g{2n}}=1-\frac{\s{2n-1}}{\g{2n}};
\end{align*}
thus, ${{\g{2n-1}}/{\g{2n}}}\rightarrow 1$. Moreover, the above equality implies
\begin{align*}
{\frac{\g{2n-1}}{\g{2n}}}=1-\frac{\s{2n-1}}{\g{2n-1}}\frac{\g{2n-1}}{\g{2n}};
\end{align*}
substituting ${\g{2n-1}}/{\g{2n}}$ appearing in the right side of this equation with the entire right side we obtain 
\begin{align}\label{rr}
{\frac{\g{2n-1}}{\g{2n}}}=1-\frac{\s{2n-1}}{\g{2n-1}}\left(1-\frac{\s{2n-1}}{\g{2n-1}}\frac{\g{2n-1}}{\g{2n}}\right)=1-\frac{\s{2n-1}}{\g{2n-1}}+\OO{\frac{|\s{2n-1}|}{\g{2n-1}^2}}.
\end{align}
From \eqref{cs}$(a)$ we obtain
\begin{align} \label{css}\Re(\zt{n})=\left(1-\frac{\s{2n-1}}{\g{2n-1}}\frac{\g{2n-1}}{\g{2n}}\right)\left(1-\frac{\omega^2}{2\,\g{2n-1}^2}+\frac{\epsilon_{n}}{2\,\g{2n-1}}\right);
\end{align}
Substituting \eqref{epsilon} and \eqref{rr} in \eqref{css} and after performing the corresponding multiplications we obtain \eqref{csa}.
Similarly, from \eqref{cs}$(b)$ we obtain
\begin{align}
\Im(\zt{n})=\frac{\omega}{\g{2n-1}}\left(1-\frac{\s{2n-1}}{\g{2n-1}}\frac{\g{2n-1}}{\g{2n}}\right)
\left(1+\frac{\s{2n-1}-\s{2n-2}}{2\,\g{2n-1}}\right);
\end{align}
substituting \eqref{rr} we obtain \eqref{sna}.
Similarly, from \eqref{sn}$(a)$ we obtain
\begin{align}\label{snn}
\Re(\ztt{n})=\frac{\omega}{2\g{2n-1}}\frac{\g{2n-1}}{\g{2n}}\left(\frac{\omega}{\,\g{2n-1}}+\frac{\epsilon_n}{\omega}\right).
\end{align}
Substituting \eqref{epsilon} and \eqref{rr} in \eqref{snn} we obtain \eqref{csta}.
Finally, from \eqref{sn}$(b)$ we obtain 
\begin{align}\label{snnn}
\Im(\ztt{n})=-\frac{\omega\left(\s{2n-1}+\s{2n-2}\right)}{2\,\g{2n-1}^2}\frac{\g{2n-1}}{\g{2n}}.
\end{align}
Substituting \eqref{rr} in \eqref{snnn} we obtain \eqref{snta}.\footnote{\emph{Mathematica} verification of Lemma~\ref{basicn} is in the file \texttt{OP/symbolic/2\_lemma\_\ref{basicn}.nb}.}
\end{proof}

We now want to obtain asymptotic representations of $\zt{n-1}$ and $\ztt{n-1}$ in terms of $\g{2n-1}$ and the finite differences. 
\begin{lemma}\label{basicm}
\begin{align}
\cs{n-1}&=1-\frac{\omega^2}{2\g{2n-1}^2}-\frac{\s{2n-3}+\s{2n-4}}{2\g{2n-1}}+\OO{\frac{\eta_n}{\g{2n-1}^2}};\label{csma}\\
\sn{n-1}&=\frac{\omega}{\g{2n-1}}\left(1+\frac{\s{2n-3}+\s{2n-4}}{2\g{2n-1}}+\frac{\ds{2n-3}+\ds{2n-4}}{\g{2n-1}}\right)+\OO{\frac{\eta_n}{\g{2n-1}^3}};\\
\cst{n-1}&=\frac{\omega^2}{2\g{2n-1}^2}+\frac{\ds{2n-4}}{2\g{2n-1}}+\OO{\frac{\eta_n}{\g{2n-1}^2}};\\
\snt{n-1}&=-\frac{\omega(\s{2n-4}+\s{2n-3})}{2\g{2n-1}^2}+\OO{\frac{\eta_n}{\g{2n-1}^3}}.\label{sntma}
\end{align}
\end{lemma}
\begin{proof}
We first obtain
\begin{align*}
\frac{\g{2n-1}}{\g{2n-3}}=1+\frac{\g{2n-1}-\g{2n-3}}{\g{2n-3}}=1+\frac{\s{2n-2}+\s{2n-3}}{\g{2n-3}}=1+\frac{\s{2n-2}+\s{2n-3}}{\g{2n-1}}\frac{\g{2n-1}}{\g{2n-3}}.
\end{align*}
Thus, also ${\g{2n-1}}/{\g{2n-3}}\rightarrow 1$ and substituting ${\g{2n-1}}/{\g{2n-3}}$ in the right side of the last equality with the entire right side of this equality, after a corresponding multiplication we obtain 
\begin{align}\label{rr3}
\frac{\g{2n-1}}{\g{2n-3}}=1+\frac{\s{2n-2}+\s{2n-3}}{\g{2n-1}}+\OO{\frac{\eta_n}{\g{2n-1}^2}}.
\end{align}
Substituting $n$ with $n-1$ in equation \eqref{csa} we obtain
\begin{align*}
\Re(\zt{n-1})&=1-\frac{\omega^2}{2\g{2n-3}^2}-\frac{\s{2n-3}+\s{2n-4}}{2\g{2n-3}}+\OO{\frac{|\s{2n-3}|}{\g{2n-3}^2}}\\
&=1-\frac{\omega^2}{2\g{2n-1}^2}\frac{\g{2n-1}^2}{\g{2n-3}^2}-\frac{\s{2n-3}+\s{2n-4}}{2\g{2n-1}}\frac{\g{2n-1}}{\g{2n-3}}+\OO{\frac{|\s{2n-3}|}{\g{2n-1}^2}}
\end{align*}
Substituting \eqref{rr3} in the above equation, after the corresponding multiplications we obtain \eqref{csma}. Other equalities are proved in an identical manner.\footnote{\emph{Mathematica} symbolic verification of Lemma~\ref{basicm} is in the file \texttt{OP/symbolic/3\_lemma\_\ref{basicm}.nb};  numerical ``confirmations'' of  Lemma~\ref{basicn} and Lemma~\ref{basicm} are in the file \texttt{OP/numerical/1\_lemmas\_\ref{basicn}\_and\_\ref{basicm}.nb}. {Of course, our numerical tests are not a part of the proofs; however, given the complexities of some of the formulas involved, they provide reassurance that nothing has been overlooked in derivations of asymptotic estimates.} Test families include five families defined by $\g{n}=(n+1)^p$ for $p=1/100,\ 1/4,\ 1/2,\ 3/4,\ 99/100$; the remaining five families are obtained from the same recursion coefficients by going a few steps backwards and then forward again; in this way one obtains non-monotonic but \emph{almost increasing} sequences; for details see any of the Mathematica files in the folder \texttt{OP/numerical}.}
\end{proof}

\section{Asymptotic behaviour of $H_j(t)$}

The following lemma follows directly from Lemma \ref{basicn} and Lemma \ref{basicm} by substitution and some simplification.
\begin{lemma}\label{ztztt}
\begin{align}\begin{split}
\zt{n}+\ztt{n}{\e}^{-\ii t}=1-\frac{\omega^2}{2\g{2n-1}^2}+\frac{\ii\,\omega}{\g{2n-1}}+\frac{\omega^2}{2\g{2n-1}^2}{\e}^{-\ii t}+\frac{\ds{2n-2}\cos t}{2\g{2n-1}}\\-\frac{\s{2n-2}+\s{2n-1}}{2\g{2n-1}}-\frac{\ii\,\ds{2n-2}\sin t}{2\g{2n-1}}+\OO{\frac{\eta_n}{\g{2n-1}^2}};\label{zt1}\end{split}
\end{align}
\begin{align}\begin{split}
\zt{n-1}-\overline{\ztt{n-1}}{\e}^{\ii t}=1-\frac{\omega^2}{2\g{2n-1}^2}+\frac{\ii\,\omega}{\g{2n-1}}-\frac{\omega^2}{2\g{2n-1}^2}{\e}^{\ii t}-\frac{\ds{2n-4}\cos t}{2\g{2n-1}}\\-\frac{\s{2n-4}+\s{2n-3}}{2\g{2n-1}}-\frac{\ii\,\ds{2n-4}\sin t}{2\g{2n-1}}+\OO{\frac{\eta_n}{\g{2n-1}^2}}.\label{ztt1}\end{split}
\end{align}\hfill$\Box$
\end{lemma}

Taking logarithms of both sides of each of the above two equations and using a first order expansion of the logarithm function we obtain the following lemma.\footnote{\emph{Mathematica}  symbolic verifications of Lemma~\ref{ztztt} and Lemma~\ref{lgztztt} are in file \texttt{OP/symbolic/4\_lemmas\_\ref{ztztt}\_and\_\ref{lgztztt}.nb}; numerical corroborations of these two lemmata are included in file \texttt{OP/numerical/2\_lemmas\_12\_13\_and\_14.nb.}}

\begin{lemma}\label{lgztztt}
\begin{align}\label{zt11}
\ln(\zt{n}+\ztt{n}{\e}^{-\ii t})=\ln\left(1-\frac{\omega^2}{2\g{2n-1}^2}+\frac{\ii\,\omega}{\g{2n-1}}+\frac{\omega^2}{2\g{2n-1}^2}{\e}^{-\ii t}\right)+\frac{\ds{2n-2}\cos t}{2\g{2n-1}}\\-\frac{\s{2n-2}+\s{2n-1}}{2\g{2n-1}}-\frac{\ii\,\ds{2n-2}\sin t}{2\g{2n-1}}+\OO{\frac{\eta_n}{\g{2n-1}^2}};
\nonumber\\
\ln(\zt{n-1}-\overline{\ztt{n-1}}{\e}^{\ii t})=\ln\left(1-\frac{\omega^2}{2\g{2n-1}^2}+\frac{\ii\,\omega}{\g{2n-1}}-\frac{\omega^2}{2\g{2n-1}^2}{\e}^{\ii t}\right)-\frac{\ds{2n-4}\cos t}{2\g{2n-1}}\label{ztt11}\\-\frac{\s{2n-4}+\s{2n-3}}{2\g{2n-1}}-\frac{\ii\,\ds{2n-4}\sin t}{2\g{2n-1}}+\OO{\frac{\eta_n}{\g{2n-1}^2}}.\nonumber
\end{align}\hfill$\Box$
\end{lemma}

To obtain an asymptotic representation of $F_n(t)$ we need the following technical lemma.
\begin{lemma}\label{twologs}
\begin{align}
\ln(|\zt{n-1}|^2-|\ztt{n-1}|^2)&=-\frac{\s{2n-4}+\s{2n-3}}{\g{2n-1}}+\OO{\frac{\eta_n}{\g{2n-1}^2}};\label{log1}\\
\ln\lambda_{n-1}&=\frac{\s{2n-4}+2\s{2n-3}+\s{2n-2}}{2\g{2n-1}}+\OO{\frac{\eta_n}{\g{2n-1}^2}}\label{log2}.
\end{align}
\end{lemma}
\begin{proof}
Substituting $n$ by $n-1$ in \eqref{bb1} we obtain
\begin{align*}
|\zt{n-1}|^2-|\ztt{n-1}|^2&=1-\frac{\s{2n-4}+\s{2n-3}}{\g{2n-2}}=
1-\frac{\s{2n-4}+\s{2n-3}}{\g{2n-1}}\left(1+\frac{\s{2n-2}}{\g{2n-2}}\right)\\
&=1-\frac{\s{2n-4}+\s{2n-3}}{\g{2n-1}}+\OO{\frac{|\s{2n-2}|}{\g{2n-1}^2}}.
\end{align*}
A first order Taylor's formula for the logarithm now yields \eqref{log1}. To prove \eqref{log2}, 
\begin{align}
&\frac{\frac{1}{\g{2n-4}}+\frac{1}{\g{2n-3}}}{\frac{1}{\g{2n-2}}+\frac{1}{\g{2n-1}}}=
\frac{\g{2n-1}}{\g{2n-3}}\frac{\g{2n-2}}{\g{2n-4}}\frac{\g{2n-3}+\g{2n-4}}{\g{2n-1}+\g{2n-2}}\nonumber\\
&\hspace*{5mm}=\left(1+\frac{\g{2n-1}-\g{2n-3}}{\g{2n-3}}\right)\left(1+\frac{\g{2n-2}-\g{2n-4}}{\g{2n-4}}\right)\nonumber\\
&\hspace*{30mm}\left(1+\frac{\g{2n-3}-\g{2n-1}+\g{2n-4}-\g{2n-2}}{\g{2n-1}+\g{2n-2}}\right)\nonumber\\
&\hspace*{5mm}=\left(1+\frac{\s{2n-3}+\s{2n-2}}{\g{2n-1}}\frac{\g{2n-1}}{\g{2n-3}}\right)\left(1+\frac{\s{2n-3}+\s{2n-4}}{\g{2n-1}}\frac{\g{2n-1}}{\g{2n-4}}\right)\nonumber\\
&\hspace*{30mm}\left(1-\frac{2\s{2n-3}+\s{2n-2}+\s{2n-4}}{2\g{2n-1}}\frac{2\g{2n-1}}{\g{2n-1}+\g{2n-2}}\right).\label{lambda1}
\end{align}
By \eqref{rr3} we have ${\g{2n-1}}/{\g{2n-3}}=1+\OO{\frac{\eta_n}{\g{2n-1}}};$ in a similar way we also obtain
\begin{align*}
\frac{\g{2n-1}}{\g{2n-4}}=1+\OO{\frac{\eta_n}{\g{2n-1}}};&&
\frac{2\g{2n-1}}{\g{2n-1}+\g{2n-2}}=1+\OO{\frac{\eta_n}{\g{2n-1}}}.&
\end{align*}
After substituting these in \eqref{lambda1} and after performing the corresponding multiplications, keeping only terms which are not necessarily absolutely summable, we obtain 
\begin{align}\label{alambda}
\lambda_{n-1}&=1+\frac{\s{2n-3}+\s{2n-2}}{\g{2n-1}}+\frac{\s{2n-3}+\s{2n-4}}{\g{2n-1}}-\frac{2\s{2n-3}+\s{2n-2}+\s{2n-4}}{2\g{2n-1}}+\OO{\frac{\eta_n}{\g{2n-1}^2}}\nonumber\\
&=1+\frac{\s{2n-4}+2\s{2n-3}+\s{2n-2}}{2\g{2n-1}}+\OO{\frac{\eta_n}{\g{2n-1}^2}}.
\end{align}
Applying now a first order approximation for the logarithm function yields  \eqref{log2}.\footnote{\emph{Mathematica } numerical confirmation of Lemma \ref{twologs} is included in file \texttt{OP/numerical/2\_lemmas\_\ref{ztztt}\_\ref{lgztztt}\_and\_\ref{twologs}.nb}.}\end{proof}

To simplify our formulas let us set
\begin{align}\begin{split}\label{fff}
f(x,t)=\ln\left(1-\frac{x^2}{2}-\ii x+\frac{x^2}{2}{\e}^{\ii t}\right)+\ln\left(1-\frac{x^2}{2}+\ii x+\frac{x^2}{2}{\e}^{-\ii t}\right)
\\-\ln\left(1-\frac{x^2}{2}-\ii x-\frac{x^2}{2}{\e}^{-\ii t}\right)-\ln\left(1-\frac{x^2}{2}+\ii x-\frac{x^2}{2}{{\e}^{\ii t}}\right);
\end{split}
\end{align}
\begin{align}\begin{split}
\label{ggg}g(x,t)=\ii \ln\left(1-\frac{x^2}{2}-\ii x+\frac{x^2}{2}{{\e}^{\ii t}}\right)-\ii \ln\left(1-\frac{x^2}{2}+\ii x+\frac{x^2}{2}{\e}^{-\ii t}\right)\\+\ii \ln\left(1-\frac{x^2}{2}-\ii x-\frac{x^2}{2}{\e}^{-\ii t}\right)-\ii \ln\left(1-\frac{x^2}{2}+\ii x-\frac{x^2}{2}{\e}^{\ii t}\right);
\end{split}
\end{align}
\begin{align}\label{hhh}
h(x,t)=&\frac{f(x,t)}{g(x,t)}.
\end{align}

\begin{lemma}\label{FG1}
\begin{align}
F_n(t)=&f\!\left(\frac{\omega}{\g{2n-1}},t\right)+\frac{(\ds{2n-4}+\ds{2n-2})\cos t}{\g{2n-1}}-\frac{\ds{2n-3}+\ds{2n-2}}{\g{2n-1}}+\OO{\frac{\eta_n}{\g{2n-1}^2}};\label{FFF}\\
G_n(t)=&g\!\left(\frac{\omega}{\g{2n-1}},t\right)-\frac{(\ds{2n-4}+\ds{2n-2})\sin t}{\g{2n-1}}+\OO{\frac{\eta_n}{\g{2n-1}^2}}.\label{GGG}
\end{align}
\end{lemma}
\begin{proof}Lemma \ref{FG1} follows from Lemma~\ref{lgztztt} and Lemma~\ref{twologs}  as well as equations \eqref{*1} and \eqref{*2} by substitutions.\footnote{\emph{Mathematica} symbolic verification of Lemma~\ref{FG1} is contained in the file \texttt{OP/symbolic/5\_lemma\_\ref{FG1}.nb}; numerical confirmation of Lemma~\ref{FG1} is included in file \texttt{OP/numerical/3\_lemma\_\ref{FG1}\_eq\_\ref{recG}\_lemma\_\ref{F/G}.nb}.} 
It is worth noting that terms $-{(\s{2n-2}+\s{2n-1})}/(2\g{2n-1})$ as well as  $-{(\s{2n-4}+\s{2n-3})}/(2\g{2n-1})$, coming from \eqref{zt11} and \eqref{ztt11} respectively, which are not necessarily absolutely summable, cancel out in $G_n(t)$; in $F_n(t)$ these two terms combine to produce $-(\ds{2n-4}+2\ds{2n-3}+\ds{2n-2})/{\g{2n-1}}$,  while term $-{(\s{2n-4}+\s{2n-3})}/{\g{2n-1}}$ coming from  $\ln(|\zt{n-1}|^2-|\ztt{n-1}|^2)$ combines with the term  ${(\s{2n-4}+2\s{2n-3}+\s{2n-2})}/{(2\g{2n-1})}$ coming from $\ln\lambda_{n-1}$, resulting in the term ${(\ds{2n-4}+\ds{2n-3})}/{\g{2n-1}}$ which is an absolutely summable term. Thus, in asymptotic representations of $F_n(t)$ and $G_n(t)$ all terms which involve finite differences are absolutely summable; such elimination of non-absolutely summable terms  is the reason why we replaced  the logarithm appearing in \eqref{largelog} with the telescopic sum \eqref{telescope}. \end{proof}

A series expansion of the right side of \eqref{GGG} with respect to $\omega/\g{2n-1}$ shows that 
\begin{align}\label{serG}
\frac{\g{2n-1}}{\omega}\,G_n(t)=4+\OO{\frac{1}{\g{2n-1}}}.
\end{align}
This, together with \eqref{sumDel} implies that 
\begin{align}\label{asdel}
\del{n-1}+\del{n}=\frac{2\omega}{\g{2n-1}}+\OO{\frac{1}{\g{2n-1}^2}}.
\end{align}
Using  \eqref{GGG} and \eqref{serG} to obtain a first order expansion of the reciprocal of  $\g{2n-1}G_n(t)$ we obtain
\begin{align*}
\frac{1}{\g{2n-1}G_n(t)}=\frac{1}{\g{2n-1}\,g\!\left(\frac{\omega}{\g{2n-1}},t\right)}+\frac{(\ds{2n-4}+\ds{2n-2})\sin t}{16\omega^2}+\OO{\frac{\eta_n}{\g{2n-1}}},
\end{align*}
which implies 
\begin{align}\label{recG}
\frac{1}{G_n(t)}=\frac{1}{g\!\left(\frac{\omega}{\g{2n-1}},t\right)}+\frac{\g{2n-1}(\ds{2n-4}+\ds{2n-2})\sin t}{16\omega^2}+\OO{\eta_n}.
\end{align}
Multiplying the corresponding sides of \eqref{FFF} and \eqref{recG} and using series expansions of ${f\!\left(\frac{\omega}{\g{2n-1}},t\right)}$ and $1/{g\!\left(\frac{\omega}{\g{2n-1}},t\right)}$ in terms containing finite differences we obtain the following lemma.\footnote{\emph{Mathematica} symbolic verifications of equation \eqref{recG} and Lemma~\ref{F/G} are in  file \texttt{OP/symbolic/6\_eq\_\ref{recG}\_and\_lemma\_\ref{F/G}.nb}; their numerical confirmations are in file \texttt{OP/numerical/3\_lemma\_\ref{FG1}\_eq\_\ref{recG}\_lemma\_\ref{F/G}.nb}.}
\begin{lemma}\label{F/G}
\begin{align}\label{f/g}
H_n(t)
=h\!\left(\frac{\omega}{\g{2n-1}},t\right)+ \frac{(\ds{2n-2}+\ds{2n-4})\cos t-\ds{2n-3}-\ds{2n-2}}{4\omega}+\OO{\frac{\eta_n}{\g{2n-1}}}.
\end{align}\hfill$\Box$
\end{lemma}

It is easy to see that $h(x,t)$ simplifies to\footnote{See file \texttt{OP/symbolic/7\_simplifying\_h(x,t).nb}.}
\begin{align*}
h(x,t)=\frac{\ln\left(1+\frac{2x^2\left(1-\frac{x^2}{2}\right)\cos t}{1-x^2\cos t-x^3\sin t+x^4\cos^2\!\!\frac{t}{2}}\right)}{2\arctan\frac{2x\left(1-\frac{x^2}{2}\right)\left(1-\frac{x}{2}\sin t\right)}{1-2x^2+x^3\sin t}};
\end{align*}
we find it interesting that such, quite a complicated function, should naturally arise in the course of our analysis of the asymptotic behaviour of $p_{2n}^2(\omega)+p_{2n+1}^2(\omega)$, given that no specific functions of any kind are present in the assumed conditions on the recurrence coefficients. 

\section{Convergence of the the sum $\mathcal{S}^\ast_n$}
\noindent We now return to the proof of convergence of $\mathcal{S}^\ast_n$ given by \eqref{s*}.
Since $\displaystyle{\lim_{n\rightarrow \infty}\frac{\g{2n-k}}{\g{2n-1}}=1}$ for any fixed $k$, 
\begin{align}\label{cons}
\sum_{n=2}^{\infty}\frac{\eta_n}{\g{2n-1}^2}&=\sum_{n=2}^{\infty}\frac{|\s{2n-1}|+|\s{2n-2}|+|\s{2n-3}|+|\s{2n-4}|}{\g{2n-1}^2}\nonumber\\
&=\OO{\sum_{n=2}^{\infty}\frac{|\s{2n-1}|}{\g{2n-1}^2}+\frac{|\s{2n-2}|}{\g{2n-2}^2}+\frac{|\s{2n-3}|}{\g{2n-3}^2}+\frac{|\s{2n-4}|}{\g{2n-4}^2}}\nonumber\\
&=\OO{\sum_{n=2}^{\infty}\frac{|\s{2n-1}|}{\g{2n-1}^2}}.
\end{align}
Thus, equations \eqref{asdel} and \eqref{cons} together with 
conditions \ref{c5} and \ref{c6} imply that  
\begin{align}
\sum_{n=2}^{\infty}\left(\frac{(\ds{2n-2}+\ds{2n-4})\cos t-\ds{2n-3}-\ds{2n-2}}{4\omega}+\OO{\frac{\eta_n}{\g{2n-1}}}\right)(\del{n-1}+\del{n})<\infty.
\end{align}
Consequently, Lemma \ref{F/G} implies that, in order to prove convergence of the sum $\mathcal{S}^\ast_n$ given by \eqref{s*}, it is sufficient to prove that 
\begin{align}\label{sprime}
\mathcal{S}^\prime&=\sum_{n=2}^{\infty} h\!\left(\frac{\omega}{\g{2n-1}},2\Te{n-1}\right)
(\del{n-1}+\del{n})<\infty.
\end{align}~\\

We now  deal with the fact that in the sum $\mathcal{S}^\prime$ given by \eqref{sprime} for every sampling point $\Te{n-1}$, the value of the parameter ${\omega}/{\g{2n-1}}$ appearing in $h\!\left({\omega}/{\g{2n-1}},2\Te{n-1}\right)$ is different. To this end, we now treat $x$ as a fixed parameter and expand the real valued $2\pi$-periodic function $h(x,t)$ into Fourier series with respect to variable $t$:
\begin{align*}
h(x,t)=\sum_{m=-\infty}^{\infty}c_{m}(x){\e}^{\ii m\,t},
\end{align*}
with $(c_{m}(x)\,:\,m\in\Zset)$ given by
\begin{equation}
c_{m}(x)=\frac{1}{2\pi}\int_{-\pi}^{\pi}{h(x, t)}\,{\e}^{-\ii m t}\dx t.\label{fint}
\end{equation}


Let us set
\begin{align}
c^n_m=c_m\left(\frac{\omega}{\g{2n-1}}\right).
\end{align}
The following Lemma will be proved in Section~\ref{fsh}. 
\begin{lemma}\label{main2} Let $c_m^j\in \Cset$ be such that 
\begin{align}\label{S*}
h\!\left(\frac{\omega}{\g{2j-1}}, t\right)&=\sum_{m=-\infty}^{\infty}c^j_{m}{\e}^{\ii m t};
\end{align}
then for every fixed $m$ coefficients $c_m^j$ have the following properties for all sufficiently large $j$:
\begin{enumerate}
\item[(i)] $c_0^j=0$ and $c_{-m}^j=\overline{c_{m}^j}$;
\item[(ii)] if $m$ is even, then $c_m^j$ satisfy ${c}^j_{m}=\ii(-1)^{\frac{m}{2}}|{c}^j_{m} |$; thus, once $m$ is fixed, then for all sufficiently large $j$, $\Re{(c_m^j)}=0$ and $\Im{(c_m^j)}$ are all of the same sign;
\item[(iii)] if $m$ is odd, then $c_m^j$ satisfy ${c}^j_{m}=(-1)^{\frac{m-1}{2}}|{c}^j_{m} |$; thus, for a fixed $m$, for all sufficiently large $j$,   $\Im{(c_m^j)}=0$ and $\Re{(c_m^j)}$ are all of the same sign;
\item[(iv)] $|c_m^j|$ are monotonic in $\omega/\g{2j-1}$ and thus form an almost decreasing sequence;
\item[(v)] $|c_m^j|=\OO{{\g{2j-1}^{-|m|}}}$.
\end{enumerate}
\end{lemma}

To prove convergence of the sum $\mathcal{S}_n^\prime$ given by \eqref{sprime}, let $j_0, m_0$ be such that ${\omega}/{4\g{j_0}}<1$ and that, according to condition  \ref{c1},  for every $j>j_0$ and every $m\geq m_0$, $\g{j+m}>\g{j}$. Let also $j_1=j_0+m_0$. We now prove that the following sum converges as $n\rightarrow\infty$:
\begin{align}
\sum_{j=j_1}^{n} h\!\left(\frac{\omega}{\g{2j-1}},2\Te{j-1}\right)
(\del{j-1}+\del{j}).
\end{align}
By \ref{c8} there exists an integer  $\kappa\geq2$ such that $\sum_{n=0}^\infty\g{n}^{-\kappa}$ converges.
Clearly, it is enough to prove convergence of $\mathcal{S}_n^h$ and $\mathcal{S}_n^l$, where
\begin{align}
\mathcal{S}_n^{h}&=\sum_{j=j_1}^{n} \sum_{|m|\geq \kappa-1}c^j_{m}{\e}^{\ii 2m\Te{j-1}}(\Delta_{j-1}+\Delta_{j});\label{lbl3}\\
\mathcal{S}_n^{l}&=\sum_{j=j_1}^{n} \sum_{|m|<\kappa-1}c^j_{m}{\e}^{\ii 2m\Te{j-1}}(\Delta_{j-1}+\Delta_{j}).\label{lbl2}
\end{align}

We first show that $\mathcal{S}_n^{h}$ converges.  Using Lemma \ref{main2}\emph{(v)} and \eqref{asdel}, we obtain 
\begin{align*}
\mathcal{S}_n^{h}\leq \sum_{j=j_1}^{n}\sum_{|m|\geq\kappa-1} |c^j_{m}| (\Delta_{j-1}+\Delta_{j})&=\OO{\sum_{j=j_1}^{n}\sum_{m=\kappa}^{\infty} \left(\frac{\omega}{4\g{2j-1}}\right)^{m}}.
\end{align*}
Since
\begin{align*}
\sum_{j=j_1}^{n}\sum_{m=\kappa}^{\infty} \left(\frac{\omega}{4\g{2j-1}}\right)^{m}&=\sum_{j=j_1}^{n}\sum_{m=0}^{\infty} \left(\frac{\omega}{4\g{2j-1}}\right)^{m+\kappa}\\&<{\sum_{j=j_1}^{n}\sum_{m=0}^{\infty} \left(\frac{\omega}{4\g{2j_0-1}}\right)^{m}\left(\frac{\omega}{4\g{2j-1}}\right)^{\kappa}}\\
&= {\sum_{m=0}^{\infty} \left(\frac{\omega}{4\g{2j_0-1}}\right)^{m}\sum_{j=j_1}^{n}\left(\frac{\omega}{4\g{2j-1}}\right)^{\kappa}} \\&=\frac{1}{1-\frac{\omega}{4\g{2j_0-1}}}\sum_{j=j_1}^{n}\left(\frac{\omega}{4\g{2j-1}}\right)^{\kappa}, 
\end{align*}
we conclude that $\mathcal{S}_n^{h}$ converges as $n\rightarrow\infty$.

Finally, to show that $\mathcal{S}_n^{l}$ converges, since $c_0^j=0$, it is enough to show that for every $m$ such that $1\leq m\leq \kappa-2$, the sum 
\begin{equation}\label{smn}
\sigma_{m,n}=\sum_{j=j_1}^{n} (c^j_{m}{\e}^{\ii 2m\Te{j-1}}+c^j_{-m}{\e}^{-\ii 2m\Te{j-1}})(\Delta_{j-1}+\Delta_{j})
\end{equation} 
converges as $n\rightarrow\infty$. Since by Lemma \ref{main2}\emph{(i)} 
\begin{align}\label{mnsum}
\sigma_{m,n}&=\sum_{j=j_1}^{n}2\Re(c^j_{m}{\e}^{\ii 2m\Te{j-1}})(\Delta_{j-1}+\Delta_{j}),
\end{align}
using Lemma \ref{main2} \emph{(ii)} and \emph{(iii)} we obtain that for all $m$, $1\leq m\leq \kappa-2$,
\begin{align}\label{sigmadef}
\sigma_{m,n}&=\begin{cases}\sum_{j=j_1}^{n}2(-1)^{\frac{m}{2}+1}|{c}^j_{m} |\sin(2m\;\Te{j-1})(\Delta_{j-1}+\Delta_{j})& \text{if\ } m \text{ is even;}\\~\\
\sum_{j=j_1}^{n}2(-1)^{\frac{m-1}{2}}|{c}^j_{m} |\cos(2m\;\Te{j-1})(\Delta_{j-1}+\Delta_{j})& \text{if\ } m \text{ is odd.}\end{cases}
\end{align}
If we now define 
\begin{align}
\tau_m(x)&=
\begin{cases}0,&\ \text{if\ } x< 2\Te{1}-\del{1};\\
2(-1)^{\frac{m}{2}+1}|{c}^{j}_{m}|\sin{m x},& \text{\ if $m$ is even and \ }2\Te{j-1}-\del{j-1}\leq x<2\Te{j-1}+\del{j};\\
2(-1)^{\frac{m-1}{2}}|{c}^{j}_{m}| \cos{m x},& \text{\ if $m$ is odd and \ } 2\Te{j-1}-\del{j-1}\leq x<2\Te{j-1}+\del{j},
\end{cases}
\end{align}
then clearly the sum $\sigma_{m,n}$ is a Riemann sum for the integral 
\begin{align}
&I_m(n)=\int_{2\Te{j_1-1}-\del{j_1-1}}^{2\Te{n-1}+\del{n}}\tau_m(x)dx,&
\end{align}
with a partition of the interval $[2\Te{j_1-1}-\del{j_1-1}, 2\Te{n-1}+\del{n}]$ into segments of the form $[2\Te{j-1}-\del{j-1},2\Te{j-1}+\del{j}]$ and with points $2\Te{j-1}$ as the sampling points for the integrand. However, in order to facilitate some estimates which we will have to make later\footnote{This is to avoid having to take a square root of the right hand sides of equations \eqref{e2} and \eqref{r3}, and thus avoid dealing with the necessary branch cuts for these square roots.}, instead we define two functions $\tau^e_m(x)$ and $\tau^o_m(x)$: 
\begin{align}
\tau_m^e(x)&=
\begin{cases}0,&\ \text{if\ } x< 2\Te{1};\\
2(-1)^{\frac{m}{2}+1}|{c}^{2j+1}_{m}|\sin{m x},&\ \text{if\ }  m \text{\ is even and \ }2\Te{2j-1}\leq x<2\Te{2j+1};\\
2(-1)^{\frac{m-1}{2}}|{c}^{2j+1}_{m}| \cos{m x},&\ \text{if\ }  m \text{\ is odd and \ } 2\Te{2j-1}\leq x<2\Te{2j+1}.
\end{cases}
\end{align}
\begin{align}
\tau_m^o(x)&=
\begin{cases}0,&\ \text{if\ }  x< 2\Te{0};\\
2(-1)^{\frac{m}{2}+1}|{c}^{2j+2}_{m}|\sin{m x},&\ \text{if\ }  m \text{\ is even and \ }2\Te{2j}\leq x<2\Te{2j+2};\\
2(-1)^{\frac{m-1}{2}}|{c}^{2j+2}_{m}| \cos{m x},&\ \text{if\ }  m \text{\ is odd and \ } 2\Te{2j}\leq x<2\Te{2j+2}.
\end{cases}
\end{align}

To obtain corresponding Riemann sums, for $\tau^e_m(x)$ we consider  a partition  of the interval $[2\Te{1}, 2\Te{2n+1}]$ into segments of the form $[2\Te{2j-1}, 2\Te{2j+1}]$ with sampling points $2\Te{2j}\in [2\Te{2j-1}, 2\Te{2j+1}]$, while for function $\tau^o_m(x)$ we consider a partition  of the interval $[2\Te{0}, 2\Te{2n}]$ into segments of the form $[2\Te{2j-2}, 2\Te{2j}]$ and sampling points $2\Te{2j-1}\in [2\Te{2j-2}, 2\Te{2j}]$. Thus, we define the following two sums:
\begin{align}
\sigma_{m,n}^e&=\begin{cases}\sum_{j=1}^{n}(-1)^{\frac{m}{2}+1}|{c}^{2j+1}_{m} |\sin(2m\;\Te{2j})(2\Delta_{2j}+2\Delta_{2j+1})& \text{if\ } m \text{ is even;}\\~\\
\sum_{j=1}^{n}(-1)^{\frac{m-1}{2}}|{c}^{2j+1}_{m} |\cos(2m\;\Te{2j})(2\Delta_{2j}+2\Delta_{2j+1})& \text{if\ } m \text{ is odd;}\end{cases}\\
\sigma_{m,n}^o&=\begin{cases}\sum_{j=1}^{n}(-1)^{\frac{m}{2}+1}|{c}^{2j}_{m} |\sin(2m\;\Te{2j-1})(2\Delta_{{2j-1}}+2\Delta_{{2j}})& \text{if\ } m \text{ is even;}\\~\\
\sum_{j=1}^{n}(-1)^{\frac{m-1}{2}}|{c}^{2j}_{m} |\cos(2m\;\Te{2j-1})(2\Delta_{2j-1}+2\Delta_{2j})& \text{if\ } m \text{ is odd;}\end{cases}
\end{align}
and also consider definite integrals
\begin{align}
&I_m^e(n)=\int_{2\Te{1}}^{2\Te{2n+1}}\tau_m^e(x)dx;&&I_m^o(n)=\int_{2\Te{0}}^{2\Te{2n}}\tau_m^o(x)dx.&
\end{align}
Clearly,  $\sigma_{m,n}^e$ and $\sigma_{m,n}^o$ are Riemann sums for $I_m^e(n)$ and $I_m^o(n)$, respectively. 

By splitting each sum $\sigma_{m,n}$ into two sums it follows that, in order to prove convergence of a sum $\sigma_{m,n}$ as $n\rightarrow\infty$ (for a fixed $m$), it is enough to prove the convergence of the sums $\sigma_{m,n}^e$ and $\sigma_{m,n}^o$. To this end, it  is enough to show that both integrals $I_m^e(n)$ and $I_m^o(n)$ converge, as well as that the errors  of approximating these integrals by the corresponding Riemann sums also converge.
 
Let us fix $m$ such that $1\leq |m|\leq \kappa-2$ and assume that $k_0,j_0$ are such that for all $k,j$ satisfying $k\geq k_0$ and $j\geq j_0$ we have $\g{j+k}>\g{j}$; then by Lemma \ref{main2} \emph{(iv)}, $|{c}^{j+k}_{m}|<|{c}^{j}_{m}|$. Since by \eqref{Te} $\Te{j}\rightarrow \infty$, $\LIM{n}I_m^e(n)$ can be represented as a sum of integrals of $\tau_m^e(x)$ between the consecutive zeros of $\tau_m^e(x)$, which are points of the form $l\pi/m$, for $l$ an integer. These integrals are alternating in sign. Since $\Delta_j=\Te{j}-\Te{j-1}\rightarrow 0$, eventually the number of points $\Te{j}$ in any interval of length $\pi/m$ will be greater that $k_0$. Thus, since $|{c}^{j+k}_{m}|<|{c}^{j}_{m}|$ for all $k>k_0$,  for all sufficiently large $x$, $|\tau_m^e(x+\pi/m)|<|\tau_m^e(x)|$, and so the absolute values of integrals of $\tau_m^e(x)$ between the consecutive zeros of $\tau_m^e(x)$ will eventually be monotonically decreasing. Consequently,  
\begin{align}
\lim_{n\rightarrow\infty}I_m^e(n)=\int_{2\Te{1}}^{\infty}\tau_m^e(x)dx<\infty,
\end{align}
and the same applies to $I_m^o(n)$. 

Finally, to conclude that $\sigma_{m,n}$ converge as $n\rightarrow\infty$, it is enough to show that the sum of errors of approximation of $I_m^e(n)$ by the Riemann sum $\sigma_{m,n}^e$ and of $I_m^o(n)$ by the Riemann sum $\sigma_{m,n}^o$, i.e., that
\begin{align}
\mathcal{E}_m(n)=\sum_{j=1}^nc_m^{j}\left(\int_{2\Te{j-2}}^{2\Te{j}}{\e}^{\ii m x}d x-{\e}^{\ii m \; 2\Te{j-1}}(2\Delta_{j-1}+2\Delta_{j})\right)
\end{align}
also converge as $n\rightarrow\infty$. We achieve such a goal by applying the same technique several times: we reduce $\mathcal{E}_m(n)$ to Riemann sums for integrals of some damped complex exponentials but with the Fourier coefficients decreased by a factor of at least ${\omega^2}/{\g{2j-1}^2}$. Thus, after a few iterations, the resulting sums will be of the form $\OO{\sum_{j=1}^n(1/{\g{2j-1}})^\kappa}$ and thus absolutely convergent.


\section{Estimating $\mathcal{E}_m(n)$}

Note that
\begin{align}\begin{split}
\int_{2\Te{j-2}}^{2\Te{j}}{\e}^{\ii m x}d x-{\e}^{\ii m \; 2\Te{j-1}}(2\Delta_{j-1}+2\Delta_{j})=\hspace*{35mm}\\
{\e}^{\ii m\; 2\Te{j-1}}\left(\frac{\ii \left({\e}^{-2\ii m \Delta_{j-1}}-{\e}^{2\ii m \Delta_{j}}\right)}{2m(\Delta_{j-1}+\Delta_{j})}-1\right)(2\Delta_{j-1}+2\Delta_{j}).\end{split}
\end{align}

From \eqref{main} and \eqref{AA} we obtain
\begin{align}\label{e0}
\mult{n}\,{\e}^{\ii \del{n}}&=\zt{n}+\ztt{n}{\e}^{-\ii 2\Te{n-1}}.
\end{align}
Taking the complex conjugate of both sides we obtain
\begin{align}\label{e1}
\mult{n}\,{\e}^{-\ii \del{n}}&=\overline{\zt{n}+\ztt{n}{\e}^{-\ii 2\Te{n-1}}}.
\end{align}
By dividing each side of \eqref{e1} by the corresponding side of \eqref{e0} we get
\begin{align}\label{e2}
{{\e}^{2\ii \del{n}}}&=\frac{\ \ \zt{n}+\ztt{n}{\e}^{-\ii 2\Te{n-1}}\ \ }
{\overline{\zt{n}+\ztt{n}{\e}^{-\ii 2\Te{n-1}}}}=\frac{\zt{n}+\ztt{n}{\e}^{-\ii 2\Te{n-1}}}
{\overline{\zt{n}}+\overline{\ztt{n}}{\e}^{\ii 2\Te{n-1}}}.
\end{align}
Similarly, substituting $n$ by $n-1$ in \eqref{nnm}  we get
\begin{align}\label{r2}
\frac{{\e}^{\ii\del{n-1}}}{\mult{n-1}}&=\frac{{\zt{n-1}}-\overline{\ztt{n-1}}{\e}^{\ii 2\Te{n-1}}}{|{\zt{n-1}}|^2-|{\ztt{n-1}}|^2}.
\end{align}
Taking the complex conjugates of both sides of \eqref{r2} produces 
\begin{align}\label{r22}
\frac{{\e}^{-\ii\del{n-1}}}{\mult{n-1}}&=\frac{\ \ \overline{{\zt{n-1}}-\overline{\ztt{n-1}}{\e}^{\ii 2\Te{n-1}}}\ \ }{|{\zt{n-1}}|^2-|{\ztt{n-1}}|^2}.
\end{align}
Dividing both sides of  \eqref{r22} by the corresponding sides of  \eqref{r2}  yields
\begin{align}\label{r3}
{\e}^{-2\ii\del{n-1}}&=\frac{\ \ \overline{{\zt{n-1}}-\overline{\ztt{n-1}}{\e}^{\ii 2\Te{n-1}}}\ \ }{{\zt{n-1}}-\overline{\ztt{n-1}}{\e}^{\ii 2\Te{n-1}}}=\frac{\overline{{\zt{n-1}}}-{\ztt{n-1}}{\e}^{-\ii 2\Te{n-1}}}{{\zt{n-1}}-\overline{\ztt{n-1}}{\e}^{\ii 2\Te{n-1}}}. 
\end{align}
Combining \eqref{e2} with \eqref{r3} we get 
\begin{align*}
{\e}^{-\ii2m\del{n-1}}-{{\e}^{\ii 2m\del{n}}}&=\left(\frac{\ \ \overline{{\zt{n-1}}}-{\ztt{n-1}}{\e}^{-\ii 2\Te{n-1}}\ \ }{{\zt{n-1}}-\overline{\ztt{n-1}}{\e}^{\ii 2\Te{n-1}}} \right)^{m}
-\left(\frac{\ \ \zt{n}+\ztt{n}{\e}^{-\ii 2\Te{n-1}}\ \ }
{\overline{\zt{n}}+\overline{\ztt{n}}{\e}^{\ii 2\Te{n-1}}} \right)^{m}.
\end{align*}

Let us define
\begin{equation}\label{Ldef}
L_n(m,t)=\ii\left(
\left(\frac{\ \ \overline{{\zt{n-1}}}-{\ztt{n-1}}{\e}^{-\ii  t}\ \ }{{\zt{n-1}}-\overline{\ztt{n-1}}{\e}^{\ii  t}} \right)^{m}
-\left(\frac{\ \ \zt{n}+\ztt{n}{\e}^{-\ii  t}\ \ }
{\overline{\zt{n}}+\overline{\ztt{n}}{\e}^{\ii  t}}\right)^{m}\right)
\end{equation}
and so  
\begin{equation}\label{Lb1}
{\e}^{\ii m 2\Te{n-1}}\left(\frac{\ii \left({\e}^{-\ii2m\del{n-1}}-{{\e}^{\ii 2m\del{n}}}\right)}{2m(\Delta_{n-1}+\Delta_{n})}-1\right)={\e}^{\ii m 2\Te{n-1}}\left(\frac{L_n(m,2\Te{n-1})}{2m(\Delta_{n-1}+\Delta_{n})}-1\right).
\end{equation}

Thus, using \eqref{sumDel},
\begin{align}\label{ER}
\mathcal{E}_m(n)=\sum_{j=1}^nc_m^j{\e}^{\ii m 2\Te{j-1}}\left(\frac{ L_n(m,2\Te{j-1})}{m\,G_n\left(2\Te{j-1}\right)}-1\right)(2\Delta_{j-1}+2\Delta_{j}).
\end{align}
We want to expand $L_n(m,t)/G_n(t)$ into Fourier series; however, 
to obtain Fourier series coefficients which are monotonic in $\g{n}$, we again need to eliminate all finite differences from $L_n(m,t)/G_n(t)$. 

\begin{lemma} \label{lm}
\begin{align}\label{fracas}
\frac{\overline{\zt{n-1}}-{\ztt{n-1}}{\e}^{-\ii t}}{\zt{n-1}-\overline{\ztt{n-1}}{\e}^{\ii t}}&
=\frac{1-\frac{\omega^2}{2\g{2n-1}^2}-\frac{\ii\,\omega}{\g{2n-1}}-\frac{\omega^2}{2\g{2n-1}^2}{\e}^{-\ii t}}
{1-\frac{\omega^2}{2\g{2n-1}^2}+\frac{\ii\,\omega}{\g{2n-1}}-\frac{\omega^2}{2\g{2n-1}^2}{\e}^{\ii t}}+\frac{\ii\,\ds{2n-4}\sin t}{\g{2n-1}}+\OO{\frac{\eta_n}{\g{2n-1}^2}};\\
\label{frac2}\frac{{\zt{n}}+{\ztt{n}}{\e}^{-\ii t}}{\overline{\zt{n}}+\overline{\ztt{n}}{\e}^{\ii t}}&=
\frac{1-\frac{\omega^2}{2\g{2n-1}^2}+\frac{\ii\,\omega}{\g{2n-1}}+\frac{\omega^2}{2\g{2n-1}^2}{\e}^{-\ii t}}
{1-\frac{\omega^2}{2\g{2n-1}^2}-\frac{\ii\,\omega}{\g{2n-1}}+\frac{\omega^2}{2\g{2n-1}^2}{\e}^{\ii t}}-\frac{\ii\,\ds{2n-2}\sin t}{\g{2n-1}}+\OO{\frac{\eta_n}{\g{2n-1}^2}}.
\end{align}
\end{lemma}

\begin{proof}
From Lemma~\ref{ztztt} we obtain by taking the first order expansions of the reciprocals\footnote{Details of the symbolic calculations involved in Lemma~\ref{lm} are in \emph{Mathematica} file \texttt{OP/symbolic/8\_lemma\_\ref{lm}.nb}; numerical corroboration of this lemma is in file \texttt{OP/numerical/4\_lemma\_\ref{lm}.nb}}
\begin{align}\begin{split}
\frac{1}{\zt{n-1}-\overline{\ztt{n-1}}{\e}^{\ii t}}=\frac{1}{1-\frac{\omega^2}{2\g{2n-1}^2}+\frac{\ii\,\omega}{\g{2n-1}}-\frac{\omega^2}{2\g{2n-1}^2}{\e}^{\ii t}}+\frac{\ds{2n-4}\cos t}{2\g{2n-1}}+\frac{\s{2n-4}+\s{2n-3}}{2\g{2n-1}}\nonumber\\+\frac{\ii\,\ds{2n-4}\sin t}{2\g{2n-1}}+\OO{\frac{\eta_n}{\g{2n-1}^2}};
\end{split}\end{align}
\begin{align}\begin{split}
\frac{1}{\overline{\zt{n}}+\overline{\ztt{n}}{\e}^{\ii t}}=\frac{1}{1-\frac{\omega^2}{2\g{2n-1}^2}-\frac{\ii\,\omega}{\g{2n-1}}+\frac{\omega^2}{2\g{2n-1}^2}{\e}^{\ii t}}-\frac{\ds{2n-2}\cos t}{2\g{2n-1}}+\frac{\s{2n-2}+\s{2n-1}}{2\g{2n-1}}-\\
\frac{\ii\,\ds{2n-2}\sin t}{2\g{2n-1}}+\OO{\frac{\eta_n}{\g{2n-1}^2}}.
\end{split}
\end{align}

Using Lemma~\ref{ztztt} again, after the corresponding multiplications and series expansion of the reciprocals appearing in terms  containing finite differences, we obtain \eqref{fracas} and \eqref{frac2}.
\end{proof}

We now define 
\begin{align}
l(m,x,t)=\ii\left(\left(\frac{1-\frac{x^2}{2}-\ii x-\frac{x^2}{2}{\e}^{-\ii t}}
{1-\frac{x^2}{2}+\ii x-\frac{x^2}{2}{\e}^{\ii t}}\right)^m-
\left(\frac{1-\frac{x^2}{2}+\ii x+\frac{x^2}{2}{\e}^{-\ii t}}{1-\frac{x^2}{2}-\ii x+\frac{x^2}{2}{\e}^{\ii t}}\right)^m\right).
\end{align}
Note that 
\begin{align}\label{conjugl}
l(-m,x,t)&=-\overline{l(m,x,t)};\\
l(m,x,\pi-t)&=\overline{l(m,x,t)}.\label{pimin}
\end{align}
Taking a first order expansion of the difference $L_n(m,t)-l\!\left(m,\frac{\omega}{\g{2n-1}},t\right)$ we obtain the following lemma.\footnote{Lemma~\ref{lm3} is symbolically verified in file \texttt{OP/symbolic/9\_lemma\_\ref{lm3}.nb} and numerically in file \texttt{OP/numerical/5\_lemmas\_\ref{lm3}\_and\_\ref{LG}.nb.}}
\begin{lemma}\label{lm3}
\begin{align}\label{ll2}
L_n(m,t)=l\!\left(m,\frac{\omega}{\g{2n-1}},t\right)-\frac{ m(\ds{2n-4}+\ds{2n-2})\sin t}{\g{2n-1}}+\OO{\frac{\eta_n}{\g{2n-1}^2}}.
\end{align}\hfill$\Box$
\end{lemma}
Multiplying the corresponding sides of \eqref{ll2} and \eqref{recG} we obtain the following lemma.\footnote{Lemma~\ref{LG} is verified in file \texttt{OP/symbolic/10\_lemma\_20.nb} and numerically in file \texttt{OP/numerical/5\_lemmas\_\ref{lm3}\_and\_\ref{LG}.nb.}}
\begin{lemma}\label{LG}
\begin{align}\label{lg}
\frac{L_n(m,t)}{{m\,G_n(t)}}=\frac{l\!\left(m,\frac{\omega}{\g{2n-1}},t\right)}{m\,g\!\left(\frac{\omega}{\g{2n-1}},t\right)}-\frac{(\ds{2n-4}+\ds{2n-2})^2\sin^2 t}{16\omega^2}+\OO{\frac{\eta_n}{\g{2n-1}}}.
\end{align}\hfill$\Box$
\end{lemma}
Since $\del{n-1}+\del{n}=\OO{1/\g{2n-1}}$, Lemma \ref{LG}, together with \eqref{asdel} and conditions \ref{c5} and \ref{c6},  imply that it  suffices to show that for every $m$ the sum
\begin{align}\label{estar}
\mathcal{E}_m^\ast(n)=\sum_{j=1}^nc_m^j{\e}^{\ii 2m\Te{j-1}}\left(\frac{ l\!\left(m,\frac{\omega}{\g{2j-1}},2\Te{j-1}\right)}{m\,g\!\left(\frac{\omega}{\g{2j-1}},2\Te{j-1}\right)}-1\right)(2\Delta_{j-1}+2\Delta_{j})
\end{align}
converges as $n\rightarrow \infty$.

We now consider $m$ and $x$ fixed parameters and expand into Fourier series with respect to variable $t$ functions 
\begin{equation}
\varepsilon_{m}(x,t)={\e}^{\ii m t}\left(\frac{ l\!\left(m,x,t\right)}{m\,g\!\left(x,t\right)}-1\right). 
 \end{equation}
Thus, with 
\begin{align}\label{cff}
f_k^{m}(x)=\frac{1}{2\pi}\int_{-\pi}^{\pi}\varepsilon_{m}(x,t) {\e}^{-\ii k t}dt=\frac{1}{2\pi}\int_{-\pi}^{\pi}\left(\frac{ l\!\left(m,x,t\right)}{m\,g\!\left(x,t\right)}-1\right){\e}^{\ii (m- k) t}dt,
\end{align}
we have
\begin{equation}
\varepsilon_{m}(x,t)=\sum_{k=-\infty}^{\infty}f_k^m(x) {\e}^{\ii k t}.
\end{equation}
We also let 
\begin{equation}
f_k^{m,n}=f_k^m\left(\frac{\omega}{\g{2n-1}}\right).
\end{equation}

The following lemma will be proved in Section~\ref{fse}.
\begin{lemma}\label{key2}
Let $f_k^{m,n}$ be such that 
\begin{equation}
\varepsilon_{m}\left(\frac{\omega}{\g{2n-1}},t\right)=\sum_{k=-\infty}^{\infty}f_k^{m,n}{\e}^{\ii k t};
\end{equation}
then for every fixed $m\neq 0$ and $k$, the coefficients $f_k^{m,n}$ satisfy the following properties for all sufficiently large $n$: 
\begin{enumerate}
\item[(i)] $f^{m,n}_{k}=\overline{f^{-m,n}_{-k}}$;
\item[(ii)] if $m$ and $k$ are of the same parity, the coefficients $f_k^{m,n}$ are real; otherwise they are purely imaginary;
\item[(iii)] the absolute values $|f_k^{\,m,n}|$ of the coefficients 
$f_k^{\,m,n}$ form an almost decreasing sequence with respect to $n$;
\item[(iv)] $|f_k^{\,m,n}|=\OO{{\g{2n-1}^{-2}}}$;
\item[(v)] if $|k|>|m|$, then  $|f_k^{m,n}|=\OO{{\g{2n-1}^{-2(|k|-|m|)}}}$. 
\end{enumerate}
\end{lemma}

The sum $\mathcal{E}_m^\ast(n)$ given by \eqref{estar} can be represented as 
\begin{align}\begin{split}
\mathcal{E}_m^\ast(n)=\sum_{j=1}^n\sum_{|k|\leq\frac{\kappa+|m|}{2}-1}c_m^j f_k^{m,j}{\e}^{\ii k 2\Te{j-1}}(2\Delta_{j-1}+2\Delta_{j})+\\
\sum_{j=1}^n\sum_{|k|>\frac{\kappa+|m|}{2}-1}c_m^j f_k^{m,j}{\e}^{\ii k t}(2\Delta_{j-1}+2\Delta_{j}).\end{split}
\end{align}
By Lemma~\ref{main2}$(v)$ and Lemma~\ref{key2}$(iv)$, 
\begin{align}\label{increase}
|c_m^jf_k^{m,j}|=\OO{{\g{2j-1}}^{-(2|k|-|m|)}}.
\end{align}
Thus, for $|k|\geq1/2(\kappa+|m|)-1$ equation \eqref{increase} implies $c_m^jf_k^{m,j}=\OO{{\g{2j-1}}^{-(\kappa-1)}}$. Consequently,  the second sum is absolutely convergent for the same reason as the sum \eqref{lbl3}. We now show that for every $m$ such that $1\leq m\leq \kappa-2$ and every $k$ such that $| k |\leq \frac{\kappa+m}{2}-1$, the sum
\begin{align}\label{mk}
\mathcal{E}_{m,k}^\ast(n)=\sum_{j=1}^n(c_m^j f_k^{m,j}{\e}^{\ii k 2\Te{j-1}}+c_{-m}^j f_{-k}^{-m,j}{\e}^{-\ii k 2\Te{j-1}})(2\Delta_{j-1}+2\Delta_{j})
\end{align}
converges as $n\rightarrow \infty$. Note that by Lemma~\ref{main2}$(i)$  and Lemma~\ref{key2}$(i)$, $c_{-m}^j f_{-k}^{-m,j}=\overline{c_m^j f_k^{m,j}}$;  by Lemma~\ref{main2}$(ii)$ and $(iii)$ and Lemma~\ref{key2}$(ii)$ the product $c_m^j f_k^{m,j}$ is either purely imaginary or real; by Lemma~\ref{main2}$(iv)$ and Lemma~\ref{key2}$(iii)$,  $|c_m^j f_k^{m,j}|$ is an almost decreasing sequence. Finally, by Lemma~\ref{main2}$(v)$ and Lemma~\ref{key2}$(iv)$ 
\begin{align}
|c_m^j f_k^{m,j}|=\OO{\g{2j-1}^{-(m+2)}}.
\end{align}
Thus, the sums $\mathcal{E}_{m,k}^\ast(n)$ are of the same form as the sum $\sigma_{m,n}$ given by \eqref{smn}, except that the coefficients $c_m^j$ are replaced by coefficients $c_m^j f_k^{m,j}$ whose absolute values are smaller by a factor of at least $\OO{\g{2j-1}^{-2}}$. Consequently,  we can repeat on all of the finitely many resulting sums the entire procedure which we applied  to $\sigma_{m,n}$,  until the Fourier coefficients become of the order of  $\g{2n-1}^{-(\kappa-1)}$ and thus the corresponding sums $\mathcal{E}_{m,k}^\ast(n)$ become of the form $\mathcal{E}_{m,k}^\ast(n)=\OO{\sum_{j=1}^n{\g{2j-1}^{-\kappa}}}$ and consequently, by Condition \ref{c8}, absolutely convergent. This concludes our proof of Theorem \ref{t1}.


\section{Fourier series expansion of $h(x,t)$}\label{fsh}

We now prove Lemma~\ref{main2}.
Note that for all real $x$ such that $|x|<1/4,$  function $h(x, t)$ is real valued and thus ${c}_{-m}(x)=\overline{{c}_m(x)}$ for all $m$. 

\begin{lemma}\label{cor-main-1}  For all real $x$ such that $|x|<1/4$,\\
(i)\ \    ${c}_0(x) =0$;\\
(ii)\ \    for all $m\neq 0$  the coefficients  ${c}_{m}(x)$ are purely imaginary for all even $m$  and real for all odd $m$.
\hspace*{\fill}$\Box$
\end{lemma}

\begin{proof}From definitions \eqref{fff} and  \eqref{ggg}  it follows directly that 
\begin{align}
f(x, t)&=-f(x,{\pi}- t);\\ 
g(x, t)&=g(x,{\pi}-t).\label{gpi}
\end{align} 
Thus, $h(x, t)=-h(x,{\pi}-t)$. Since $h(x,t)$ is real,  $c_0=\overline{c_0}$ and using a substitution $u=\pi-t$, \eqref{fint} implies 
\begin{align}
c_{m}(x)&=-\frac{1}{2\pi}\int_{-\pi}^{\pi}{h(x, \pi-t)}\,{\e}^{-\ii m t}\dx t 
= -\frac{1}{2\pi}\int_{0}^{2\pi}{h(x, u)}\,{\e}^{-\ii m (\pi-u)}\dx u\nonumber\\&=(-1)^{m+1}\overline{c_{m}},
\end{align}
which implies both claims of the lemma.
\end{proof}

\begin{lemma}\label{mono}
For every fixed $m\neq 0$, $|{c}_{m}(x)|$ is monotonic in $x$ in a sufficiently small neighbourhood of $0$; thus, for all $m\neq0$, 
$|c_m^j|$ form an almost decreasing sequence with respect to $j$.
\end{lemma}
\begin{proof}To prove the above Lemma we consider function 
\[c_m(z)=\frac{1}{2\pi}\int_{-\pi}^{\pi}{\e}^{-\ii m t}{h(z,t)}\dx t\] 
on the compact set $U=\{(z,t)\,:\,|z|\leq 1/4,\,|t|\leq \pi\}\subset \Cset\times\Rset$, with $h(z,t)$ in the form given by \eqref{hhh}, together with definitions \eqref{fff} and \eqref{ggg}.  For every fixed $t$ function  ${\e}^{-\ii m t}h(z,t)$  is analytic on the disc $|z|\leq 1/4$ in the complex plane. Thus, for every closed contour $C\subset U$,  Fubini's and Cauchy's theorems imply
\[\oint_{C}\int_{-\pi}^{\pi}{\e}^{-\ii m t}{h(z,t)}\dx t\dx z=\int_{-\pi}^{\pi}\oint_{C}{\e}^{-\ii m t}{h(z,t)}\dx z\dx t = 0.\]
Consequently, by Morera's theorem, function $c_m(z)$ is analytic on the disc $|z|\leq 1/4$. Note that for real $x$ such that \ $|x|<1/4$ the values of $\ii^{m-1}c_m(x)$ and $\ii^{m-1}c_m^\prime(x)=\frac{d}{d x}(\ii^{m-1} c_m(x))$ are real, and if there were no neighbourhood of $0$ in which $\ii^{m-1}c_m(x)$ is monotonic, $c_m^\prime(x)$ would change its sign infinitely many times in every neighbourhood of $0$ and thus also have infinitely many zeros in the set $\{z\,:\, |z|\leq 1/4\}$. Since $c_m^\prime(z)$ is also analytic on that set, this would imply that there $c_m^\prime(z)$ is identically equal to 0. However, as we will see, equation \eqref{cf3} shows that this is impossible. 
\end{proof}
We now want to establish the asymptotic behaviour of $c_m(x)$ as $x\rightarrow 0$.  
If $x$ is real and $|x|<1/4$  we can combine the first logarithm in \eqref{fff} and \eqref{ggg} with the fourth and the second with the third and  obtain
\begin{align}
f(x,t)&=\ln\frac{1-\frac{x^2}{2}+\ii x +\frac{x^2}{2}{\e}^{-\ii t}}
{1-\frac{x^2}{2}-\ii x -\frac{x^2}{2}{\e}^{-\ii t}}+
\ln\frac{1-\frac{x^2}{2}-\ii x +\frac{x^2}{2}{\e}^{\ii t}}
{1-\frac{x^2}{2}+\ii x -\frac{x^2}{2}{\e}^{\ii t}};\label{fs}\\
g(x,t)&=\ii\left(-\ln\frac{1-\frac{x^2}{2}+\ii x +\frac{x^2}{2}{\e}^{-\ii t}}
{1-\frac{x^2}{2}-\ii x -\frac{x^2}{2}{\e}^{-\ii t}}+
\ln\frac{1-\frac{x^2}{2}-\ii x +\frac{x^2}{2}{\e}^{\ii t}}
{1-\frac{x^2}{2}+\ii x -\frac{x^2}{2}{\e}^{\ii t}}\right).\label{gs}
\end{align}
Note that we can multiply by ${\e}^{\ii t}$ both the denominator and the numerator of the fraction in the first of the two logarithms appearing in \eqref{fs} and \eqref{gs} thereby eliminating ${\e}^{-\ii t}$. Thus, if we define
\begin{align}
f^{\ast}(x,z)&=\ln\frac{\left(1-\frac{x^2}{2}+\ii x \right)z+\frac{x^2}{2}}
{\left(1-\frac{x^2}{2}-\ii x \right)z-\frac{x^2}{2}}+
\ln\frac{1-\frac{x^2}{2}-\ii x +\frac{x^2}{2}z}
{1-\frac{x^2}{2}+\ii x -\frac{x^2}{2}z};\\
g^{\ast}(x,z)&=-\ii\ln\frac{\left(1-\frac{x^2}{2}+\ii x\right)z +\frac{x^2}{2}}
{\left(1-\frac{x^2}{2}-\ii x\right)z -\frac{x^2}{2}}+
\ii\ln\frac{1-\frac{x^2}{2}-\ii x +\frac{x^2}{2}z}
{1-\frac{x^2}{2}+\ii x -\frac{x^2}{2}z};\\
h^{\ast}(x,z)&=\frac{f^{\ast}(x,z)}{g^{\ast}(x,z)},
\end{align}
then  for $|x|<1/4$ and all $t$ we have $h(x,t)={h^{\ast}\!\!\left(x,{\e}^{\ii t}\right)}$ and
\begin{align*}
c_m(x)&=\frac{1}{2\pi}\int_{-\pi}^{\pi}{h(x,t)}\,{\e}^{-\ii m t}\dx t=\frac{1}{2\pi}\oint_{|z|=1}{h^{\ast}(x,z)}z^{-m}\frac{\dx z}{\ii z}\\&=\frac{1}{2\pi\ii}\oint_{|z|=1}{  z^{-m-1}h^{\ast}(x,z)}\dx z.
\end{align*}

We assume that $x$ is a fixed parameter such that $|x|<1/4$ and look for the singularities of $z^{-m-1}h^{\ast}(x,z)$.\footnote{\emph{Mathematica} calculations of  cuts, poles and the residue at the pole of $z^{-m-1}h^{\ast}(x,z)$ are in file \texttt{OP/symbolic/11\_cuts\_and\_poles.nb.}}
Considering  the logarithms appearing in $f^{\ast}(x,z)$ and $g^{\ast}(x,z)$ and letting
\begin{align*}
&w_1=-\frac{x^2}{2\left(1+\frac{x^4}{4}\right)}\left(1-\frac{x^2}{2}-\ii x\right);&&
w_2=\frac{x^2}{2\left(1+\frac{x^4}{4}\right)}\left(1-\frac{x^2}{2}+\ii x\right),&
\end{align*}
we obtain for the numerator of the fraction inside the first logarithm
\begin{align}
&\left(1-\frac{{x}^2}{2}+\ii{x}\right)z+\frac{{x}^2}{2}=0 && \text{if and only if}\ \ z=w_1,&
\end{align} 
while setting its denominator to zero produces
\begin{align}
&\left(1-\frac{{x}^2}{2}-\ii x\right)z-\frac{{x}^2}{2}=0 && \text{if and only if}\ \ z=w_2.&
\end{align} 
Let $z$ be an arbitrary complex number such that $z\neq w_1$ and $z\neq w_2$, and let $a=\Re{(z)}, b=\Im{(z)}$; then\begin{align}\label{long1}
&\frac{\left(1-\frac{x^2}{2}+\ii x \right)(a+\ii b)+\frac{x^2}{2}}
{\left(1-\frac{x^2}{2}-\ii x \right)(a+\ii b)-\frac{x^2}{2}}\nonumber\\
&=\frac{\left({\left(1-\frac{x^2}{2}+\ii x \right)(a+\ii b)+\frac{x^2}{2}}\right)
\left({\left(1-\frac{x^2}{2}-\ii x \right)(a-\ii b)-\frac{x^2}{2}}\right)}{\left| \left(1-\frac{x^2}{2}-\ii x \right)(a+\ii b)-\frac{x^2}{2} \right|^2}\nonumber\\
&=\frac{(a^2+b^2)(1-2x^2)+bx^3+\frac{1}{4}(a^2+b^2-1)x^4+\ii\,x\left(a^2+b^2-\frac{b}{2}x\right)(2-x^2)}{\left| \left(1-\frac{x^2}{2}-\ii x \right)(a+\ii b)-\frac{x^2}{2} \right|^2}.
\end{align}
Thus, for $0<x<1/4$ the imaginary part of this fraction is zero just in case $a^2+b^2-\frac{b}{2}x=0$, i.e. for $b=\frac{1}{4}\left(x+\sqrt{x^2-16a^2}\right)$ or $b=\frac{1}{4}\left(x-\sqrt{x^2-16a^2}\right)$. If $0<x<1/4$ and $b=\frac{1}{4}\left(x+\sqrt{x^2-16a^2}\right)$ then the real part of the numerator in the last fraction of \eqref{long1} is equal to
\[\frac{1}{32}\left((x^2(4-8x^2+x^4)+x(4+x^4)\sqrt{x^2-16a^2}\right)\]
which is  positive for all $a$ for which the value of this expression is real, i.e., for $|a|<\frac{|x|}{4}$.  If $0<x<1/4$ and $b=\frac{1}{4}\left(x-\sqrt{x^2-16a^2}\right)$  then the real part of the numerator in the last fraction of \eqref{long1} is equal to
\[\frac{1}{32}\left((x^2(4-8x^2+x^4)-x(4+x^4)\sqrt{x^2-16a^2}\right)\]
which is negative for $a$ such that $\Re(w_1)<a<\Re(w_2)$. Thus, we make a cut in the complex plane which is an arc $z(a)=a+\ii \frac{1}{4}\left(x-\sqrt{x^2-16a^2}\right)$ with end points $w_1$ and $w_2$, passing through the origin; see Figure~\ref{contour}. Thus, the origin is not an isolated pole of $z^{-m-1}h^{\ast}(x,z)$ for $m>0$. Note that for $z$ lying on such an arc we have
\begin{equation}\label{normw}
|z|^2=a^2+\left(\frac{1}{4}(x-\sqrt{x^2-16a^2})\right)^2=\frac{1}{8}x\left(x-\sqrt{x^2-16a^2}\right),
\end{equation}
which increases monotonically in $a^2$ and thus attains a maximum for $a^2=\Re(w_1)^2=\Re(w_2)^2$, 
which for $0<x<1/4$ gives $ |z|^2=x^4/(4+x^4)<x^4/4$. Thus, the entire arc is contained inside a disc $\{z\,:\,|z|\leq x^2/2\}$. Such containment of the cut is yet another benefit of  pairing $\mult{n}$ with $\mult{n-1}$ and $\del{n}$ with $\del{n-1}$. \\

For the fraction inside the second logarithm appearing in $f^\ast(x,z)$ and $g^\ast(x,z)$ we obtain for its numerator and denominator, respectively
\begin{align}
&{1-\frac{x^2}{2}-\ii x +\frac{x^2}{2}z}=0 && \text{if and only if}\ \ z=v_1,\ \ \ \text{for} && v_1=1-\frac{2}{x^2}+\ii \frac{2}{x};&\\
&{1-\frac{x^2}{2}+\ii x -\frac{x^2}{2}z}=0 &&\text{if and only if}\ \ z=v_2,\ \ \ \text{for} && v_2=-1+\frac{2}{x^2}+\ii \frac{2}{x}.&
\end{align} 
Representing again $z$ as $z=a+\ii b$, we obtain that 
\begin{align}
&\frac{1-\frac{x^2}{2}-\ii x+\frac{x^2}{2}(a+\ii b)}
{1-\frac{x^2}{2}+\ii x-\frac{x^2}{2} (a+\ii b)}\nonumber\\
&\hspace*{5mm}=\frac{
\left(1-\frac{x^2}{2}-\ii x+\frac{x^2}{2}(a+\ii b)\right)
\left(1-\frac{x^2}{2}-\ii x-\frac{x^2}{2} (a-\ii b)\right)}{\left| 1-\frac{x^2}{2}+\ii x-\frac{x^2}{2} (a+\ii b) \right|^2}\nonumber\\
&\hspace*{5mm}=\frac{1-2x^2+bx^3-\frac{1}{4}(a^2+b^2-1)x^4+\ii\,x\left(\frac{b}{2}x-1\right)(2-x^2)}{\left| 1-\frac{x^2}{2}+\ii x-\frac{x^2}{2} (a+\ii b) \right|^2}.
\end{align}
Thus, the imaginary part of the above fraction is equal to zero just in case $b=2/x$, and for such $b$ the real part of the fraction is equal to
\[1-x^2+\frac{1}{4}(1-a^2)x^4\]
which is negative if either $a<\Re(v_2)$ or $a>\Re(v_1)$. Consequently, we make two cuts in the complex plane which are horizontal half lines  passing through $v_1$ and $v_2$ respectively; see Figure~\ref{contour}.
\begin{figure}[t]
   \centering
   \includegraphics[width=4.5in]{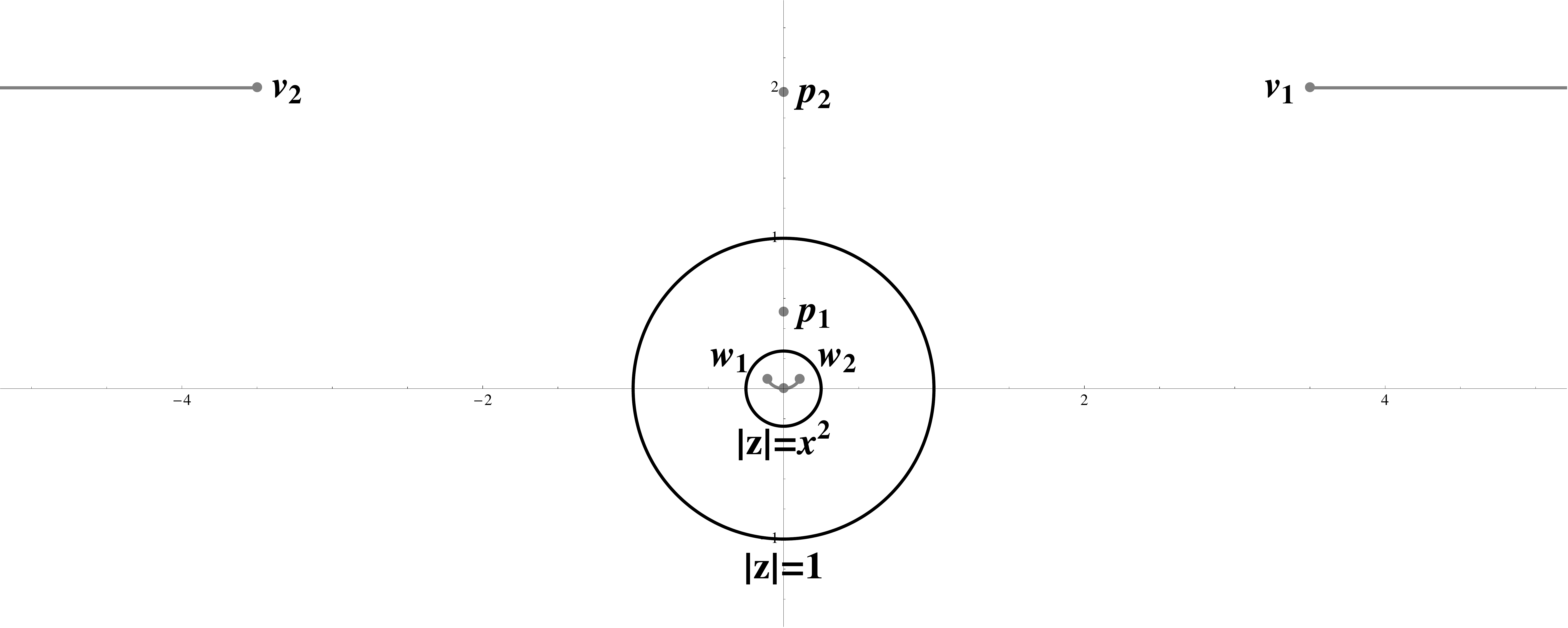}
   \caption{Integration contours with cuts and poles}
   \label{contour}
\end{figure}

The last remaining possible singularity can only occur when $g^\ast(x,z)=0$. 
Let us set 
\begin{align}
\tilde{g}(x,z)=\ii\ln\left(\frac{\left(1-\frac{x^2}{2}-\ii x\right)z -\frac{x^2}{2}}{\left(1-\frac{x^2}{2}+\ii x\right)z +\frac{x^2}{2}}\ \frac{1-\frac{x^2}{2}-\ii x +\frac{x^2}{2}z}
{1-\frac{x^2}{2}+\ii x -\frac{x^2}{2}z}\right);
\end{align}
then $g^\ast(x,z)=\tilde{g}(x,z)+2k\pi$ for some $k$ such that $|k|\leq 1$. If $g^\ast(x,z)=0$ for some $z$, then $\tilde{g}(x,z)+2k\pi=0$. However, since $|\tilde{g}(x,z)|\leq \pi$ we get that $k=0$ and  
$\tilde{g}(x,z)=0$. On the other hand, if $\tilde{g}(x,z)=0$, then  $g^\ast(x,z)=2k\pi$. Since the absolute value of the imaginary part of the first logarithm in $g^\ast(x,z)$ is smaller or equal to $\pi$ and the absolute value of the imaginary part of the second logarithm is strictly smaller than $\pi$ whenever 
$0<x<1/4$ and $|z|\leq 1$, we obtain that for such values of $x$ and $z$ we again have $k=0$ and $g^\ast(x,z)=0$. Thus, for  $0<x<1/4$ and $|z|\leq 1$ we have $g^\ast(x,z)=0$ if and only if $\tilde{g}(x,z)=0$, i.e., if and only if 
\begin{align}
\frac{\left(1-\frac{x^2}{2}-\ii x\right)z -\frac{x^2}{2}}{\left(1-\frac{x^2}{2}+\ii x\right)z +\frac{x^2}{2}}\ \frac{1-\frac{x^2}{2}-\ii x +\frac{x^2}{2}z}
{1-\frac{x^2}{2}+\ii x -\frac{x^2}{2}z}=1.
\end{align}
The solutions to this equation are
\begin{align}
&pl_1(x)=\ii \frac{x}{2\left(1+\sqrt{1-\frac{x^2}{4}}\right)}; &&pl_2(x)=\ii \frac{2\left(1+\sqrt{1-\frac{x^2}{4}}\right)}{x}.&
\end{align}
For $|x|<1/4$ pole $pl_2$ lies outside the unit disc $\{z\,:\, |z|\leq1\}$, while pole $pl_1$ lies inside the unit disc  but outside the disc $\{z\,:\, |z|\leq x^2\}$. Thus, we have
\begin{equation}\label{cf1}
\oint_{|z|=1}z^{-m-1}h^{\ast}(x,z)\dx z=\oint_{|z|=x^2}  z^{-m-1}h^{\ast}(x,z)\dx z+2\pi\ii \res{z^{-m-1}h^{\ast}(x,z)}{ pl_1}.
\end{equation}
Equations \eqref{fint} and \eqref{cf1} yield
\begin{align}\label{cf2}
c_m(x)&=\frac{1}{2\pi\ii}\oint_{|z|=x^2}  z^{-m-1}h^{\ast}(x,z)\dx z+\res{z^{-m-1}h^{\ast}(x,z)}{ pl_1}\\
&=\frac{1}{2\pi}\int_{-\pi}^{\pi}  x^{-2m}{\e}^{-\ii m t}h^{\ast}(x,x^{2}{\e}^{\ii t})\dx t+\res{z^{-m-1}h^{\ast}(x,z)}{ pl_1(x)}.\label{cf22}
\end{align}
Let us set $g_z^\prime(x,z)={\partial}g^\ast(x,z)/{\partial z}$; then direct calculations show that\footnote{These calculations as well as the residue evaluation  are also contained in file \texttt{OP/symbolic/11\_cuts\_and\_poles.nb.}}
\begin{align}
f^\ast(x,pl_1(x))&=2\ln\left(1-2x^2+\frac{x^4}{2}-2\ii x\sqrt{1-\frac{x^2}{4}}\left(1-\frac{x^2}{2}\right)\right);\label{ff}\\
g_z^\prime(x,pl_1(x))&=-8\ii\left(1-\frac{x^2}{2}\right)\sqrt{1-\frac{x^2}{4}}\left(1+\sqrt{1-\frac{x^2}{4}}\right),\label{dg}
\end{align}
and we obtain that for all $x$ such that $|x|<1/4$,
\begin{align*}
\res{z^{-m-1}h^\ast(x,z)}{ pl_1(x)}&=pl_1(x)^{-m-1}\frac{f^\ast(x,pl_1(x))}{g_z^\prime(x,pl_1(x))}.
\end{align*}
One can verify that the real part of the logarithm in \eqref{ff} is zero; thus, after some simplification, we obtain
\begin{align}\label{res1}
\res{z^{-m-1}h(x,z)}{ pl_1(x)}
&=\frac{{\ii}^{-m-1}x^{-m-1}}{2^{-m+1}\left(1+\sqrt{1-\frac{x^2}{4}}\right)^{-m}}\;\frac{\arctan\frac{2x\sqrt{1-\frac{x^2}{4}}
\left(1-\frac{x^2}{2}\right)}{1-2x^2+\frac{x^4}{4}}}
{\left(1-\frac{x^2}{2}\right)\sqrt{1-\frac{x^2}{4}}}\nonumber\\
&={\ii}^{-m-1}\left(\frac{x}{4}\right)^{-m}+\OO{x^{-m+2}}.
\end{align}
Since $h^\ast(x,x^{2}{\e}^{\ii t})$ is continuous for $(x,t)\in [0,1/4]\times [-\pi,\pi]$, for all $m\leq -1$ we have
\begin{align}\label{int1}
\left|\frac{1}{2\pi}\int_{-\pi}^{\pi}  x^{-2m}{\e}^{-\ii m t}h^\ast(x,x^{2}{\e}^{\ii t})\dx t\right|&<\frac{ x^{-2m}}{2\pi}\int_{-\pi}^{\pi} \left|h^\ast(x,x^{2}{\e}^{\ii t})\right|\dx t =\OO{x^{-2m}}.
\end{align}
Consequently, from \eqref{cf22}, \eqref{res1} and \eqref{int1} we obtain that for $m\geq1$
\begin{align}\label{cf3}\begin{split}
c_{-m}(x)&={\ii}^{m-1}\left(\frac{x}{4}\right)^{m}+\OO{x^{m+2}};\\
c_m(x)&=\overline{c_{-m}(x)}={(-\ii)}^{m-1}\left(\frac{x}{4}\right)^{m}+\OO{x^{m+2}}.
\end{split}
\end{align}
This, together with Lemma \ref{cor-main-1}, proves $(ii),(iii)$ and $(iv)$ of Lemma \ref{main2}.


\section{Fourier series of $\varepsilon_{m}(x,t)$}\label{fse}

We now prove Lemma~\ref{key2}.  Note that for $0<x<1/4$ equation \eqref{ggg} implies that $\overline{g(x,t)}=g(x,t)$; this and \eqref{conjugl} imply that 
for all $m\neq 0$, all $k$ and all sufficiently large $n$, 
\begin{equation}\label{mnk}
f^{m,n}_{k}=\overline{f^{-m,n}_{-k}}.
\end{equation}
\begin{lemma}\label{realb}
For all $m\neq 0$, all $k$ and all sufficiently large $n$, the coefficients $f_k^{m,n}$ are real if $m$ and $k$ are of the same parity and purely imaginary otherwise.
\end{lemma}
\begin{proof}Using \eqref{pimin}, \eqref{gpi}  and a substitution $u=\pi-t$, we obtain
\begin{align*}
f_{k}^{m}(x)&=\frac{1}{2\pi}\int_{-\pi}^{\pi}\left(\frac{l(m, x,t)}{m\,g(x,t)}-1\right)\,{\e}^{\ii (m-k) t}\dx t \\
&=\frac{1}{2\pi}\int_{-\pi}^{\pi}\left(\frac{\overline{l(m, x, \pi-t)}}{m\,g(x,\pi-t)}-1\right)\,{\e}^{\ii (m-k) t}\dx t \\
&= \frac{1}{2\pi}\int_{0}^{2\pi}\left(\frac{\overline{l(m, x, u)}}{m\,g(x,u)}-1\right)\,{\e}^{\ii (m-k) (\pi-u)}\dx u\\
&=(-1)^{m-k}\;\overline{f_{k}^{m}(x)}.
\end{align*}
\end{proof}

\begin{lemma}
For all $m\neq 0$, all $k$ and all sufficiently large $n$, the absolute values $|f_k^{\,m,n}|$ of the coefficients 
$f_k^{\,m,n}$ form an almost decreasing sequence.
\end{lemma}

\begin{proof} The coefficients $f_k^m(x)$ can be proven to be monotonic in $x$ in a sufficiently small neighbourhood of 0 in the same way as this was done for $c_m(x)$. Thus, since $\g{2n-1}$ are almost increasing, $f_k^{\,m,n}$ form an almost decreasing sequence.
\end{proof}

\begin{lemma}\label{lm26}
For all $m\neq 0$, all $k$ and all sufficiently large $n$,
\begin{align}\label{bk2}
&f_k^{\,m,n}=\OO{{\g{2n-1}^{-2}}}.
\end{align}
\end{lemma}

\begin{proof}
A second order expansion with respect to $x$ shows that\footnote{\emph{Mathematica} verification of Lemma~\ref{lm26} is in the file \texttt{OP/symbolic/12\_lemmas\_\ref{lm26}\_and\_\ref{resple}.nb}.} 
\begin{align}
\frac{1-\frac{x^2}{2}-\ii x -\frac{x^2}{2}{\e}^{-\ii t}}{1-\frac{x^2}{2}+\ii x -\frac{x^2}{2}{\e}^{\ii t}}=1-2\ii x -(2-\ii \sin t)x^2+\OO{x^3};\\
\frac{1-\frac{x^2}{2}+\ii x +\frac{x^2}{2}{\e}^{-\ii t}}{1-\frac{x^2}{2}-\ii x -\frac{x^2}{2}{\e}^{\ii t}}=1+2\ii x -(2+\ii \sin t)x^2+\OO{x^3}.
\end{align}
This implies 
\begin{align*}
&l(m,x,t)=\\
&\ii \left(\left(1-2\ii x -(2-\ii \sin t)x^2+\OO{x^3}\right)^m-\left(1+2\ii x -(2+\ii \sin t)x^2+\OO{x^3}\right)^m\right)\\
&=\ii \left(1-2\ii m x -m(2m-\ii \sin t)x^2-(1+2\ii m x -m (2m +\ii \sin t)x^2)+\OO{x^3}\right)\\
&=4m x-2m x^2\sin t+\OO{x^3}.
\end{align*}
A second order expansion of $g(x,t)$ with respect to $x$ also produces 
\begin{align*}
m\,g(x,t)&=4m x-2m x^2\sin t+\OO{x^3}.
\end{align*}
The last two equations imply
\[
\frac{l\!\left(m,x,t\right)}{m\,g\!\left(x,t\right)}-1=\OO{x^2},
\] 
which implies \eqref{bk2}.
\end{proof}

\begin{lemma}\label{resple} For all $m\neq0$, all $k$ and all sufficiently large $n$,
\[f_k^{m,n}=\OO{{\g{2n-1}^{-2(|k|-|m|)}}}. \]
\end{lemma}
\begin{proof}
Let us define 
\begin{align}
l^\ast(m,x,z)=\ii\left(\left(\frac{1-\frac{x^2}{2}-\ii x-\frac{x^2}{2}z^{-1}}
{1-\frac{x^2}{2}+\ii x-\frac{x^2}{2}z}\right)^m-
\left(\frac{1-\frac{x^2}{2}+\ii x+\frac{x^2}{2}z^{-1}}{1-\frac{x^2}{2}-\ii x+\frac{x^2}{2}z}\right)^m\right);
\end{align}
then $l(m,x,t)=l^\ast\!\left(m,x,{\e}^{\ii t}\right)$, and  
\begin{align}
&f_k^{\,m,n}=\frac{1}{2\pi}\oint_{|z|=1}\left(\frac{ l^\ast\!\left(m,x,z\right)}{m\,g^\ast\!\left(x,z\right)}-1\right)z^{m- k}\frac{dz}{\ii z}.
\end{align}

We have already found the singularities and branch cuts for $g^\ast(x,z)$; we also saw that if $g^\ast(x, z)=0$, $x<1/4$ and $|z|<1$, then $z=pl_1(x)$. However, a direct substitution shows that $l^\ast(m,x,pl_1(x))=0$.\footnote{See file \texttt{OP/symbolic/12\_lemmas\_\ref{lm26}\_and\_\ref{resple}.nb}.}  Thus, $z=pl_1(x)$ is a removable singularity of ${l^\ast(m,x,z)}/{g^\ast(x,z)}$ and consequently integration over the unit circle can be replaced by integration over a circle of radius $x^2$, thus obtaining 
\begin{align}
f_k^{m}(x)=\frac{1}{2\pi}\int_{-\pi}^{\pi}\left(\frac{ l^\ast\!\left(m,x,x^2{\e}^{\ii t}\right)}{m\,g^\ast\!\left(x,x^2{\e}^{\ii t}\right)}-1\right)\left(x^2{\e}^{\ii t}\right)^{m- k}dt. 
\end{align}
This is easily seen to imply $|f_{k}^{m}(x)|=\OO{x^{2(m-k)}}$. Thus, for  $k<0$ we obtain $f_{k}^{m}(x)=\OO{x^{2(|k|-|m|)}}$ and \eqref{mnk} implies that the same is true for positive $k$ as well; consequently,
$f_k^{m,n}=\OO{\g{2n-1}^{-2(|k|-|m|)}}$. 
\end{proof}


\section{Consequences of Theorem \ref{t1}}

\begin{Corollary}[\ref{corollary1}]
Let $\gamma_n$ be as in Theorem~\ref{t1}; then the limits below exist and satisfy
\begin{equation*}
\lim_{n\rightarrow \infty}
\frac{\sum_{k=0}^np^2_k(\omega)}{\sum_{k=0}^n\frac{1}{\gamma_k}}=\frac{1}{2}\lim_{n\rightarrow \infty}\g{n}(p^2_{n}(\omega)+p^2_{n+1}(\omega))
\end{equation*}
and convergence is uniform on every compact set.
\end{Corollary}

\begin{proof}
Using Theorem  \ref{t1} and equality \eqref{even}, let $L(\omega)>0$ be such that\footnote{Numerical simulations regarding the convergence and equality of the two limits mentioned in this corollary are in file \texttt{OP/numerical/5-OP-asymptotics.nb}.}
\begin{equation*}
\frac{1}{2}\LIM{n}\g{n}(p_{n}^2(\omega)+p_{n+1}^2(\omega))={\lim_{n\rightarrow \infty}{\frac{p^2_{2n}(\omega)+p^2_{2n+1}(\omega)}{\frac{1}{\g{2n}}+\frac{1}{\g{2n+1}}}}}=L(\omega).
\end{equation*}
For odd $n$, $n=2k+1$, since $\sum_{j=0}^n1/\g{j}$ diverges, the Stolz-Ces\`{a}ro theorem immediately implies that 
\begin{equation*}
\frac{\sum_{j=0}^{2k+1}p_j^2(\omega)}{\sum_{j=0}^{2k+1}\frac{1}{\g{j}}}=\frac{\sum_{j=0}^{k}(p_{2j}^2(\omega)+p_{2j+1}^2(\omega))}{\sum_{j=0}^{k}\left(\frac{1}{\g{2j}}+\frac{1}{\g{2j+1}}\right)}\rightarrow L(\omega).
\end{equation*}
For even $n$ we observe that 
\begin{align*}
\frac{\sum_{j=0}^{2k+2}p_j^2(\omega)}{\sum_{j=0}^{2k+2}\frac{1}{\g{j}}}=
\frac{\sum_{j=0}^{2k+1}p_j^2(\omega)}{\sum_{j=0}^{2k+1}\frac{1}{\g{j}}}\;
\left(1-\frac{\frac{1}{\g{2k+2}}}{\sum_{j=0}^{2k+2}\frac{1}{\g{j}}}\right)+
\frac{p_{2k+2}^2(\omega)}{\sum_{j=0}^{2k+2}\frac{1}{\g{j}}}.
\end{align*}
Our conditions on the recurrence coefficients and the case for odd $n$ imply that the first summand converges to $L(\omega)$ and that the second summand satisfies 
\begin{align*}
0\leq \frac{p_{2k+2}^2(\omega)}{\sum_{j=0}^{2k+2}\frac{1}{\g{j}}}< \frac{\g{2n+2}(p_{2k+2}^2(\omega)+p_{2k+3}^2(\omega))}{\sum_{j=0}^{2k+2}\frac{1}{\g{j}}}\rightarrow 0.
\end{align*}
\end{proof}

\begin{Corollary}[\ref{corollary2}]If 
 $\gamma_n=(n+1)^p$ for some $p$ such that $0<p<1$, then \[
0<\lim_{n\rightarrow\infty}\frac{\sum_{j=0}^np^2_j(\omega)}{(n+1)^{1-p}}<\infty.\]
\end{Corollary}
\begin{proof}
Follows from the previous Lemma and the fact that in this case  
\begin{equation*}
\sum_{k=0}^{n}\frac{1}{\g{k}}=\OO{\sum_{k=0}^{n}(k+1)^{-p}}=\OO{\int_{1}^{n+1}x^{-p}{d}x}=\OO{\frac{(n+1)^{1-p}-1}{1-p}}.\label{simpg}
\end{equation*}
\end{proof}


\section{Hilbert Spaces Associated with Orthonormal Polynomials}

We now present an application of Corollary~\ref{corollary1}; in fact, this application was the author's sole motivation  for the present work, because it  proves his conjecture from \cite{I1} under additional assumptions. We denote by $\dd_t$ the operator of differentiation  with respect to variable $t$, and in the remaining part of this paper we will consider functions  $f:\Rset\rightarrow \Cset$ whose real and imaginary parts are infinitely differentiable functions of a \emph{real variable}. 
The set of such functions will be denoted by $\Dset$. Note that if $f(t)\in\Dset$,  then also 
$\overline{f(t)}\in \Dset$ and $D_t[\overline{f(t)}]=\overline{D_t[f(t)]}$.

Assuming that the families of orthogonal polynomials we consider satisfy our conditions \ref{c2}-\ref{c6},
we define a corresponding family of linear differential operators $\K{n}_t$ by
\begin{equation}\label{k}
\K{n}_t=(-{\ii})^{n}
p_{\scriptstyle{n}}
\left(\ii\;\frac{{\rm d}}{{\rm d} t}\right).
\end{equation}
Such operators have real coefficients and satisfy
\begin{equation}
\K{n}_t[{\e}^{\ii\omega t}]=
{\ii}^np_{\scriptstyle{n}}(\omega)\,{\e}^{\ii\omega t}.
\end{equation}
It is easy to see that such operators satisfy the recurrence
\begin{equation}\label{three-term}
\g{n}\,\K{n+1}_t=\dd_t\circ
\K{n}_t+\g{n-1}\, \K{n-1}_t,
\end{equation}
with the same coefficients $\gamma_n>0$ as in \eqref{poly}. As for orthonormal polynomials, setting $\g{-1}=1$ and $\K{-1}_t[f(t)]\equiv 0$ the above recurrence is valid for all $n\geq 0$.
The following lemma shows that  differential operators $\K{n}_t$ which correspond to polynomials $\PP{n}{\omega}$ have a property which is analogous to the Christoffel--Darboux equality for orthogonal polynomials, 
\begin{equation}\label{CDP}
(\omega-\sigma)\sum_{k=0}^{n}
\PP{{k}}{\omega}\PP{{k}}{\sigma}
= \gamma_n
(\PP{{n+1}}{\omega}\PP{{n}}{\sigma}-\PP{{n+1}}{\sigma}
\PP{{n}}{\omega}),
\end{equation}

\begin{lemma}\label{cd}
For all $f,g\in\Dset$ and all $n\in\Nset$,
\begin{equation}\label{C-D}
 \dd_t\left[\sum_{m=0}^{n} \K{m}_t[f(t)]\,\K{m}_t[{g(t)}]\right]= {\g{n}}\,
(\K{n+1}_t[f(t)]\,\K{n}_t[g(t)]+\K{n}_t[f(t)]\,\K{n+1}_t[g(t)]).
\end{equation}
\end{lemma}
\begin{proof} Use \eqref{three-term} to form a telescopic sum, or, alternatively, by induction on $n$, using  \eqref{three-term} in the induction step, and its instance for $n=0$ for the base case of induction. 
\end{proof}
When applying operators $\K{n}_t$ to a function $f(t)$ of a single variable $t$ we will omit  index $t$ and write instead just $\K{n}[f(t)]$. 
More details about operators $\K{n}$ can be found in \cite{I1},\cite{IGT}.

\begin{definition} Assume that the recurrence coefficients satisfy conditions \ref{c2}-\ref{c6};
we denote by $\CC$ the vector space
of functions $f(t)\in\Dset$ such that the sequence of corresponding functions $(\beta_{n}^{f}(t)\,:\,n\in\Nset)$ defined by 
\begin{equation}\label{tau}
\beta_{n}^{f}(t)=\g{n} (|\K{n}[f(t)]|^2+ |\K{n+1}[f(t)]|^2)
\end{equation}
converges uniformly on every compact interval $I\subset\Rset$.
\end{definition}

\begin{lemma}\label{FG}
If $f(t)\in\CC$ then $(\beta_{n}^{f}(t)\,:\,n\in\Nset)$ converges to a constant function, \\
$\LIM{n}\beta_{n}^{f}(t)=L$. Moreover, if we define 
\begin{equation}\label{nu}
\nu_{n}^{f}(t)=\frac{\sum_{k=0}^{n} |\K{k}[f(t)]|^2}{\sum_{k=0}^{n}\frac{1}{\gamma_k}},
\end{equation}
then $\nu_{n}^{f}(t)$ converges to a constant function $L/2$, uniformly on every compact interval.
\end{lemma}
\begin{proof}
Assume $\beta_{n}^{f}(t)\rightarrow l(t)$, uniformly on every compact interval $I\subset\Rset$. 
Just as in the proof of Corollary \ref{corollary1}, the Stolz-Ces\`{a}ro theorem implies that $\nu_{n}^{f}(t)\rightarrow l(t)/2$. 
We now observe that  \eqref{C-D} with $g(t)=\overline{f(t)}$ yields
\begin{align*}
\left|\frac{d}{d t}\nu_{n}^{f}(t)\right|&\leq\frac{\g{n}\,
(|\K{n+1} [f(t)]\,\K{n} [\overline{f(t)}]|+|\K{n} [f(t)]\,\K{n+1} [\overline{f(t)}]|)}{\sum_{k=0}^{n}\frac{1}{\gamma_k}}\\
&\leq
\frac{\g{n}
(|\K{n} [f(t)]|^2+|\K{n+1} [f(t)]|^2)}{\sum_{k=0}^{n}\frac{1}{\gamma_k}}
\end{align*}
Since  $\g{n}(|\K{n} [f(t)]|^2+|\K{n+1} [f(t)]|^2)$ converges to a finite limit $l(t)$ uniformly on every finite interval  and  since $\sum_{k=0}^{n}{1}/{\gamma_k}$ diverges, ${d}/{d t}\;\nu_{n}^{f}(t)$ converges to $0$ uniformly on every finite interval as well. Thus,  for some $L\geq 0$, $\beta_{n}^{f}(t)$ and $\nu_{n}^{f}(t)$ converge to constant functions $L$ and $L/2$ respectively.
\end{proof}

\begin{corollary}\label{30}
Let $\mathcal{B}_0\subset \mathcal{B}$ be the set of all $f\in\mathcal{B}$ such that 
\[\LIM{n}\g{n}(|\K{n} [f(t)]|^2+|\K{n+1} [f(t)]|^2)=0;\] 
then in the quotient space $\mathcal{B}_2=\mathcal{B}/\mathcal{B}_0$, we can introduce a norm with the following expressions
whose value does not depend on $t$:
\begin{align}
\|f\|=\LIM{n}\sqrt{\frac{\g{n}}{2}(|\K{n} [f(t)]|^2+|\K{n+1} [f(t)]|^2)}=\LIM{n}\sqrt{\frac{\sum_{k=0}^{n} |\K{k}[f(t)]|^2}{\sum_{k=0}^{n}\frac{1}{\gamma_k}}}
\end{align}
\end{corollary}

\begin{lemma}\label{prop1} Let $f,g\in\CC$ and let 
\begin{equation}\label{sigma}
\sigma^{fg}_n(t)=\frac{\sum_{k=0}^{n}
\K{k}[f(t)]{\K{k}[\overline{g(t)}]}}{\sum_{k=0}^{n}\frac{1}{\gamma_k}}.
\end{equation}
If  for some $t=t_0$ the sequence $\sigma^{fg}_n(t_0)$ converges as $n\rightarrow\infty$, then $\sigma^{fg}_n(t)$ converges for all $t$ to a constant function.
\end{lemma}

\begin{proof} 
Using Lemma \ref{cd} we get 
\begin{align*}
\left|\frac{d}{d t}\sigma_{n}^{fg}(t)\right|&\leq\frac{\g{n}\,
(|\K{n+1} [f(t)]\,\K{n} [\overline{g(t)}]|+|\K{n} [f(t)]\,\K{n+1} [\overline{g(t)}]|)}{\sum_{k=0}^{n}\frac{1}{\gamma_k}}\\
&\leq
\frac{\g{n}
(|\K{n} [f(t)]|^2+|\K{n+1} [f(t)]|^2+|\K{n} [g(t)]|^2+|\K{n+1} [g(t)]|^2)}{2\sum_{k=0}^{n}\frac{1}{\gamma_k}}.
\end{align*}
Since $f,g\in\CC$, the numerator converges uniformly on every finite interval; since the denominator diverges, $\left|\frac{d}{d t}\sigma_{n}^{f}(t)\right|$ converges uniformly to zero on every finite interval. Thus, if $\sigma^{fg}(t_0)$ converges for some $t=t_0$ it converges for all $t$ to a constant function. 
\end{proof}

\begin{theorem}\label{HA}
If the recurrence coefficients satisfy conditions \ref{c2}-\ref{c6}, then the complex exponentials $e_{\omega}(t)={\e}^{\ii \omega t}$ belong to the associated space $\CC_2$ and  for every $e_{\omega}(t)={\e}^{\ii \omega t}$ and $e_{\sigma}(t)={\e}^{\ii \sigma t}$ such that $\omega\neq \sigma$ we have 
\begin{align*}
 \LIM{n}\frac{\sum_{k=0}^{n} \K{k}_{t}[{\e}^{\ii \omega t}]  \K{k}_{t}[\overline{{\e}^{\ii \sigma t}}] }
{\sum_{j=0}^{n}\frac{1}{\gamma_j}}= 0.
\end{align*}
\end{theorem}
\begin{proof}
Let $\omega> 0$;  since
\[
 |\K{n}_{t}[{\e}^{\ii \omega t}]|^2 = |{\ii}^n p_{n}(\omega){\e}^{\ii \omega t}|^2= p_{n}^2(\omega)
\]
and thus 
\begin{equation}{\g{n}( |\K{n}_{t}[{\e}^{\ii \omega t}]|^2+ |\K{n+1}_{t}[{\e}^{\ii \omega t}]|^2)}=
{\g{n}(p_{n}^2(\omega)+p_{n+1}^2(\omega))},\label{fnorm}
\end{equation}
our Corollary \ref{corollary1} implies that the sequence 
${\g{n}( |\K{n}_{t}[{\e}^{\ii \omega t}]|^2+ |\K{n+1}_{t}[{\e}^{\ii \omega t}]|^2)}$
converges to a positive finite limit uniformly in $t\in\Rset$ and uniformly in $\omega$ from any compact interval. Thus, $0<\|{\e}^{\ii \omega t}\|<\infty$ and ${\e}^{\ii \omega t} \in \CC_2$. Let $\omega\neq \sigma$. 
Note that 
\begin{align}\label{KKsin}
\sum_{k=0}^{n} \K{k}_{t}[{\e}^{\ii \omega t}]  \K{k}_{t}[{\e}^{-\ii \sigma t}] &=\sum_{k=0}^{n}  p_{k}(\omega)p_{k}(\sigma){\e}^{\ii (\omega-\sigma) t}.
\end{align}
By the Christoffel Darboux equality \eqref{CDP} we have
\begin{align}\label{cde}
\sum_{k=0}^{n}  p_{k}(\omega)p_{k}(\sigma)=\frac{\gamma_{n}}{\omega-\sigma}
(\PP{{n+1}}{\omega}\PP{{n}}{\sigma}-\PP{{n+1}}{\sigma}
\PP{{n}}{\omega}).
\end{align}
Thus, 
\begin{align*}
\left|\sum_{k=0}^{n} \K{k}_{t}[{\e}^{\ii \omega t}]  \K{k}_{t}[{\e}^{-\ii \sigma t}]\right| &=\left|\sum_{k=0}^{n}  p_{k}(\omega)p_{k}(\sigma)\right| \\
&\leq \frac{\gamma_{n}}{\omega-\sigma}
(|p_{n+1}(\omega)p_{n}(\sigma)|+|p_{n+1}(\sigma)
p_{n}(\omega)|)\\
&\leq \frac{\gamma_{n}}{2(\omega-\sigma)}
(p_{n+1}^2(\omega)+p_{n}^2(\sigma)+p_{n+1}^2(\sigma)+
p_{n}^2(\omega)).
\end{align*}
Since $\gamma_{n}
(p_{n+1}^2(\omega)+p_{n}^2(\omega))$ and $\gamma_{n}
(p_{n+1}^2(\sigma)+p_{n}^2(\sigma))$ both converge to a finite limit and $\Sigma_{k=0}^n1/\g{k}$ diverges as $n\rightarrow\infty$,  we obtain
\begin{align*}
 \LIM{n}\frac{\sum_{k=0}^{n} \K{k}_{t}[{\e}^{\ii \omega t}]  \K{k}_{t}[{\e}^{-\ii \sigma t}] }
{\sum_{k=0}^{n}\frac{1}{\gamma_k}}= 0.
\end{align*}
\end{proof}
Thus, on the vector space of all finite linear combinations of complex exponentials $\LIM{n}\sigma_n^{fg}(t)$ is independent of $t$ and it defines an inner product which makes all complex exponentials of distinct frequencies mutually orthogonal. Such an inner product can now be extended to various spaces relevant to signal processing, such as spaces of functions of the form 
$f(t)=\sum_{k=0}^\infty q_k{\e}^{\ii \omega_k t}$ where $\sum_{k=0}^\infty |q_k|<\infty$ and for some $B>0$ all $\omega_k$ satisfy $|\omega_k|<B$. Such signals have bounded amplitude and finite bandwidth; they are important, for example, for representation of speech signals. The values of  $\gamma_n(|\K{n}[f(t)]|^2+|\K{n+1}[f(t)]|^2)$ for various $n$ are measures of ``local energy" around an instant $t$ of a signal of infinite total energy (in the usual $L_2$ sense); increasing $n$ reduces such localisation.  Operators $(\K{k}\,:\,k\in\Nset)$ can also be used for frequency estimation of multiple sinusoids in the presence of coloured noise; see \cite{IGF}; they originated in the course of author's design of a pulse-width modulation switching power amplifier. 

From purely mathematical perspective, it remains to be seen if the method presented here can be applied to  non-symmetric families of orthogonal polynomials or families with bounded recurrence coefficients. For a general family of positive definite orthonormal polynomials $p_n(\omega)$ we have   
\begin{equation}\label{polyn}
{\g{n}}p_{n+1}(\omega)=({\omega}+\beta_n) p_{n}(\omega)-{\g{n-1}}\, p_{n-1}(\omega).
\end{equation}
It is easy to see that $a_n$, $b_n$ and $E_n$ can be defined so that \eqref{main1} and \eqref{polar} hold and that again $|E_n|=\prod_{j=0}^n|a_j+b_j\e^{-2\ii\Phi_{j-1}}|$. If we define $\arg z$ to satisfy $0\leq \arg z<2\pi$ with the corresponding cut along the positive reals, then for bounded recurrence coefficients such that $\g{n}\rightarrow \gamma$ and $\beta_n/\gamma_n\rightarrow \rho$  it is easy to see that $|\Im{(a_n)}|>|b_n|$ holds just in case\footnote{These calculations can be found in  \url{http://www.cse.unsw.edu.au/~ignjat/diff/OP2.zip}.}
\begin{equation}\label{condit}
 \left|\frac{\omega}{\gamma}-\rho\right|<2. 
 \end{equation}
Thus, the classical results of Paul Nevai, see Theorem 7 on page 23 of \cite{PN},  imply that our argument proving \eqref{td}, i.e., $\del{n}=\arg\bigl(\zt{n} + \ztt{n}{\e}^{-2\ii \Te{n-1}}\bigr)$ holds for all  $\omega$ in the support of the continuous part of the corresponding measure of orthogonality. As we have mentioned before, Nevai's results from \cite{PN} (formula (16) on page 140), imply that in such a case $p_{2n}^2(\omega)+p_{2n+1}^2(\omega)$ cannot converge; however, perhaps there is some hope that our method might be used to show directly that $\sum_{n=0}^{\infty}p_n^2(\omega)/(n+1)$ converges for such $\omega$, under some weak assumption on $\g{n}$ and $\beta_n$.\\

In the case of unbounded recurrence coefficients, assuming only $\g{n}\rightarrow \infty$,  $\s{n}=\g{n+1}-\g{n}\rightarrow 0$ and $\beta_n/\gamma_n\rightarrow \rho$, condition \eqref{condit} reduces to $|\rho|<2$. Thus, we obtain that $|\Im{a_n}|>|b_n|$ holds just in case $|\rho|<2$.  To the authors' big surprise, extensive numerical tests indicate that it is also just in case $|\rho|<2$ that the ratio ${\sum_{k=0}^np^2_k(\omega)}/{\sum_{k=0}^n{1}/{\gamma_k}}$ converges to a positive finite limit $L(\omega)$; the value of $\g{n}(p^2_{n}(\omega)+p^2_{n+1}(\omega))$ appears to be oscillating around $L(\omega)$, contained between two envelopes which converge to two horizontal straight lines. Thus, we finish this paper with the following conjecture.\footnote{Recently, Doron Lubinsky and the author have evaluated the limits from Corollary \ref{corollary1} for a class of even exponential weights \cite{IL} and even more recently, Grzegorz  \'{S}widerski and  Bartosz  Trojan have improved  Corollary \ref{corollary1}, proving it under weaker assumptions \cite{ST}; Grzegorz  \'{S}widerski also improved the results from \cite{IL} by evaluating these limits in a general setup \cite{S}.}
\begin{conjecture}
Assume that a family of orthonormal polynomials $p_n(\omega)$ satisfies recurrence \eqref{polyn}, coefficients $\g{n}$ are unbounded and satisfy conditions \ref{c2}-\ref{c6} and that  $\left|\LIM{n}{\beta_n}/{\g{n}}\right|<2.$ Then ${\sum_{k=0}^np^2_k(\omega)}/{\sum_{k=0}^n{1}/{\g{k}}}$ converges to a positive finite limit $L(\omega)$. Thus, in such a case the asymptotic growth rate of ${\sum_{k=0}^np^2_k(\omega)}$ does not depend on ${\beta_n}$ (but the value of $L(\omega)$ does). 
\end{conjecture}


\bibliographystyle{amsplain}

\end{document}